\documentclass[a4paper,leqno]{article}

\usepackage[latin1]{inputenc}
\usepackage{amsmath,amssymb,amsthm,
	verbatim,
	fancyhdr,amsfonts
}

\usepackage[format=plain,labelfont=it,textfont=it]{caption}
\usepackage{needspace}
\usepackage{float}
\usepackage{tikz}
\usetikzlibrary{decorations.pathreplacing}

\usepackage[pdftex]{lscape}
\usepackage{amsmath}
\usepackage{amsthm}
\usepackage{amsfonts,amsmath,amssymb}
\usepackage{enumerate,verbatim,color}
\usepackage{epsfig}
\usepackage{textcomp}
\usepackage{graphicx}
\usepackage{amsfonts}
\usepackage{mathrsfs}
\usepackage{upgreek}
\usepackage{mathtools}
\usepackage{pdfpages}
\usepackage{fullpage}

\usepackage[right=2cm,left=2cm,top=2cm,bottom=2cm]{geometry}

\newcommand{\virgolette}[1]{``#1''}
\newtheorem{teorema}{Theorem}[section]

\newtheorem{lemma}[teorema]{Lemma}
\newtheorem{prop}[teorema]{Proposition}
\newtheorem{osss}[teorema]{Remark}

\theoremstyle{definition}
\newtheorem{defi}[teorema]{Definition}
\theoremstyle{remark}
\newtheorem*{assumptionI}{\bf Assumptions I}
\newtheorem*{assumptionII}{\bf Assumptions II}
\newtheorem*{boundaries}{\bf Vertex boundary conditions}
\newtheorem*{notations}{\bf Graph boundary conditions}

\newcommand{\iii}{{\, \vert\kern-0.25ex\vert\kern-0.25ex\vert\, }}

\DeclareMathOperator{\spn}{span}
\newcommand{\ffi}{\varphi}

\newcommand{\AL}{\mathcal{A}}

\newcommand{\CC}{\mathcal{C}}
\newcommand{\LL}{\mathcal{L}}
\newcommand{\MM}{\mathcal{M}}

\newcommand{\NN}{\mathcal{N}}
\newcommand{\KK}{\mathcal{K}}

\newcommand{\Di}{\mathcal{D}}

\newcommand{\Z}{\mathbb{Z}}
\newcommand{\N}{\mathbb{N}}
\newcommand{\R}{\mathbb{R}}
\newcommand{\Q}{\mathbb{Q}}
\newcommand{\C}{\mathbb{C}}

\newcommand{\Hi}{\mathscr{H}}

\newcommand{\Gi}{\mathscr{G}}

\newcommand{\dd}{\partial}

\newcommand{\ra}{\rangle}
\newcommand{\la}{\langle}

\newcommand{\G}{\Gamma}

\date{}
\title{Bilinear quantum systems on compact graphs: well-posedness and global exact controllability}

\author{Alessandro Duca
	\\ \ \\
	{\small  Universit\'e de Lorraine, CNRS, INRIA, IECL, F-54000 Nancy, France}\\
	{\small \texttt{alessandro.duca@inria.fr}}\\
}

\begin{document}

\maketitle
\begin{abstract}
A major application of the mathematical concept of graph in 
quantum mechanics is to model networks of electrical
wires or electromagnetic wave-guides. In this paper, we address the dynamics of a 
particle trapped on such a 
network in presence of an external electromagnetic field. We study the controllability of the motion when the 
intensity of the field changes over time and plays the role of control. From a mathematical point of view, the 
dynamics of 
the particle is modeled by the so-called bilinear 
Schr\"odinger equation defined on a graph representing the network. The main purpose of this work is to extend the 
existing theory for bilinear quantum systems on bounded intervals to the framework of graphs. To this end, we 
introduce a suitable mathematical 
setting where to address the controllability of the equation from a theoretical point 
of view. More precisely, we determine assumptions on the network and on the potential field ensuring its
global exact controllability in suitable spaces. Finally, we discuss two applications of our results and their 
practical implications to two specific
problems involving a star-shaped network and a tadpole graph. 
\end{abstract}

\noindent


\noindent

\section{Introduction}\label{intro}

During the last decades, graph type models (Figure 
\ref{fig:1}) have been extensively studied in the literature for the modeling of 
phenomena arising in science, social 
sciences 
and
engineering. Applications to the quantum mechanics include the study of the dynamics of free
electrons in organic molecules starting from the seminal work \cite{198} (see also \cite{172, 201, 209, 777, 
182}), the 
superconductivity in granular and 
artificial materials \cite{121212}, acoustic and electromagnetic wave-guides networks in \cite{112, 181}, etc. 
\begin{figure}[H]
	\centering
	\includegraphics[width=\textwidth-150pt,height=30pt]{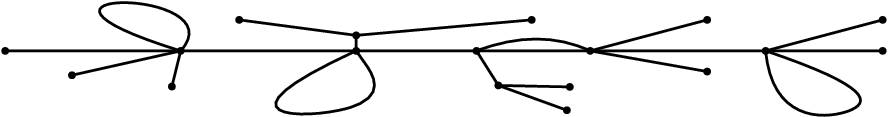}
	\caption{A compact graph is a one-dimensional domain composed by finite vertices (points) connected by edges 
(segments) of finite lengths.}\label{fig:1}
\end{figure}

The aim of this paper is to study the dynamics of a particle trapped on a 
network of wires, or wave-guides, in presence of an electromagnetic 
external field. We assume that the particle is subjected to zero resistance when it crosses the nodes of the 
network. 
The intensity of the 
external field is a function of the time and it plays the role of control. The dynamics of 
the particle is driven by the bilinear Schr\"odinger equation in $L^2(\Gi,\C)$ where $\Gi$ is the graph 
modeling the network:
\begin{equation}\label{mainx1}\tag{BSE}\begin{split}
\begin{cases}
i\dd_t\psi(t)=A\psi(t)+u(t)B\psi(t),\ \ \ \ \ \ \ \ &t\in(0,T),\\
\psi(0)=\psi_0,\ &\ \ \ \ T>0.\\
\end{cases}
\end{split}
\end{equation}
The term $u(t)B$ in the \eqref{mainx1} represents the external field acting on the system. The bounded symmetric 
operator $B$ describes the action of the field and $u\in L^2((0,T),\R)$ its intensity. The operator $A=-\Delta$ is 
a self-adjoint Laplacian equipped with 
suitable boundary conditions (see the next paragraphs for further details).

The mathematical analysis of self-adjoint operators on networks was 
addressed in \cite{209} by Ruedenber and Scherr (see also \cite{204}). There, they studied the dynamics of 
specific electrons in the conjugated double-bounds organic molecules. In such a context, some electrons behave 
as if they were trapped on a
network of wave-guides. The graphs are obtained 
as the
idealization of these structures in the limit where the diameter of their section is much smaller 
than the length. A similar approach was developed by Saito in \cite{211, 212} where the graphs are obtained as 
``shrinking'' domains. For analogous ideas, we refer to the papers \cite{208, 202}.

\medskip

A natural question on the \eqref{mainx1} of high relevance for the mentioned applications is whether it is 
possible to control 
the motion of the 
particle through the network by varying the intensity of the external field.
From a mathematical point of view, we wonder if for any couple of quantum states $\psi_0$ and $\psi_1$, there 
exists a control $u$ so that the solution of the \eqref{mainx1} with initial state $\psi_0$ reaches, or at least 
approaches, 
$\psi_1$.

A peculiarity of the bilinear quantum systems as \eqref{mainx1} is that their exact
controllability can not be ensured in the Hilbert space where the dynamics is defined when $B$ is a 
bounded operator. This 
is a consequence of the results developed by Ball, 
Mardsen and Slemrod in the work on bilinear systems $\cite{ball}$. 

For the bilinear Schr\"odinger equation 
\eqref{mainx1} on the interval $\Gi=(0,1)$, a successful 
strategy was developed by Beauchard in $\cite{be1}$ which addresses the exact controllability in suitable 
subspaces of $L^2((0,1),\C)$. Beauchard studied the equation with $A$ the Dirichlet 
Laplacian $-\Delta_D$ on $(0,1)$ and proved the local exact controllability in 
$D(|\Delta_D|^\frac{3}{2})$. This method was later exploited in 
\cite{laurent,morgane1,morganerse2,mio2,mio1}.

\medskip

Even though the global exact controllability of the bilinear Schr\"odinger equation  \eqref{mainx1} on $\Gi=(0,1)$ 
is well-established, 
the result 
on networks is still an open problem to the best of our knowledge.
{\setlength{\leftmargini}{0.4cm}\begin{itemize}
\item Firstly, it is not clear which is 
the ``good" space where to consider the dynamics. When $A$ is a self-adjoint Laplacian on a 
compact graph, it is not easy to characterize how the boundary conditions defining the 
domains $D(|A|^\frac{s}{2})$ are affected by the variation of the 
parameter $s>0$.
\item From a spectral point of view, even though $A$ admits compact resolvent (see Remark 
\ref{compact_resolvent}), its ordered eigenvalues $(\lambda_k)_{k\in\N^*}$
are not explicit up to very specific situations, unlike the Dirichlet Laplacian $-\Delta_D$ on 
$(0,1)$ 
(we refer to Remark 
\ref{autovalori espliciti} for further details on the subject). In general, a more complicated structure of the 
networks entails increased difficulties in determining the spectral behavior.

\item Finally, the techniques developed in the works 
$\cite{be1,laurent,mio1,mio2,morgane1}$ can not be directly applied for the \eqref{mainx1} on compact graphs 
$\Gi$.  Indeed, the only sure fact on $(\lambda_k)_{k\in\N^*}$ is that the 
uniform
spectral gap $\inf_{{k\in\N^*}}|\lambda_{k+1}-\lambda_k|> 0$ can be only guaranteed when $\Gi=(0,1)$. This 
hypothesis 
is crucial for the techniques developed in the mentioned works.

\end{itemize}}

The purpose of this work is threefold. First, we introduce a suitable mathematical framework where to consider 
the dynamics of the \eqref{mainx1} and we characterize the Sobolev's spaces 
$D(|A|^\frac{s}{2})$ with $s \geq 0$. Second, we study 
the well-posedness of the 
\eqref{mainx1} in such domains for very specific values of the parameter $s$. 
Third, we develop a new technique leading to the global exact controllability of the \eqref{mainx1} in these 
spaces.

Our main result yields the global exact contrability for bilinear quantum systems on network under fairly general
hypotheses on the spectrum $A$ and on the potential $B$. We infer the validity of the mentioned spectral 
assumptions for suitable networks shaped as: star graphs, double-ring graphs, tadpole 
graphs or two-tails tadpole graphs (Figure \ref{int}). In the 
following paragraph, we provide an explicit potential field $B$ acting on a star-shaped network so that the 
controllability is 
guaranteed. In Section \ref{tadpole}, we treat another specific problem on a network modeled by a
tadpole graph.

\begin{figure}[H]
	\centering
	\includegraphics[width=\textwidth-100pt,height=30pt]{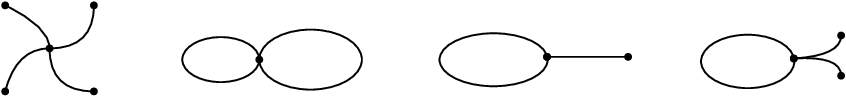}
	\caption{The figure respectively represents a star graph, a double-ring graph, a tadpole graph and a two-tails 
tadpole 
graph.}\label{int}
\end{figure}

\vspace{4mm}

\needspace{4mm}
{\noindent \bf \underline{Star-sha}p\underline{ed network}}

\vspace{2mm}

\noindent
Let us consider a network 
of $3$ branches connected in one node.  We can assume it is a structure of 
electrical wires or a system of wave-guides representing the so-called 
``branching points'' in the conjugated 
double-bounds organic molecules (see \cite{209} by Ruedenber and Scherr). 
We 
represent the network with a 
star-graph $\Gi$ composed by $3$ edges 
$\{e_1,e_2,e_3\}$. 
We denote by $v$ the internal vertex connecting all the edges of $\Gi$. We 
parametrize each $e_j$ 
with a coordinate going from $0$ to its length $L_j$ in $v$ (see Figure 
\ref{parametrizzazione}). 

\begin{figure}[H]
	\centering
	\includegraphics[width=\textwidth-100pt]{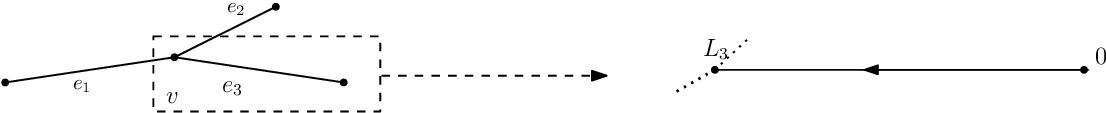}
	\caption{The figure shows the parametrization of a star graph with $3$ edges.}\label{parametrizzazione}
\end{figure}

We consider a particle trapped on such a 
network  and we represent it by a quantum state 
$\psi=(\psi^1,\psi^2,\psi^3)\in 
L^2(\Gi,\C):=\prod_{j=1}^3L^2(e_j,\C),$ where each $\psi^j:e_j\rightarrow \C$ describes the probability of
particle to be located in the edge $e_j$. We assume that an external field acts on the network as a 
sufficiently regular potential field $\mu$ localized on the 
branch represented by $e_1$ and with time-dependent intensity. We model the field as $u(t)B$ where
\begin{align}\label{potential_intro}u\in 
L^2((0,T),\R),\ \ \ \ \ \  \ \ \ B:\psi=(\psi^1,\psi^2,\psi^3)\in L^2(\Gi,\C)\longmapsto 
\big(\mu(x)\psi^1,0,0\big) .\end{align} 
The dynamics of the particle in the time $(0,T)$ is modeled by the
bilinear Schr\"odinger equation in 
$L^2(\Gi,\C)$
\begin{equation}\label{mainxS}\begin{split}\begin{cases}
i\dd_t\psi^1(t,x)=-\dd_x^2\psi^1(t,x)+u(t)\mu(x)\psi^1(t,x),\ \ \ \ \ \ \ \ &t\in(0,T),\ x\in (0,L_1),\\
i\dd_t\psi^2(t,x)=-\dd_x^2\psi^2(t,x),\ \ \ \ \ \ \ \ &t\in(0,T),\ x\in (0,L_2),\\
i\dd_t\psi^3(t,x)=-\dd_x^2\psi^3(t,x),\ \ \ \ \ \ \ \ &t\in(0,T),\ x\in (0,L_3),\\
\end{cases}\end{split}
\end{equation}
endowed with the following boundary conditions
\begin{equation}\label{mainxS_boundaries}\begin{split}
\psi^1(L_1)=\psi^2(L_2)=&\ \psi^3(L_3),  \ \ \ \ \ \  \ \ \ \ \ \
\dd_x\psi^1(L_1)+\dd_x\psi^2(L_2)+\dd_x\psi^3(L_3)=0,\\
&\psi^1(0)=\psi^2(0)=\psi^3(0)=0.\\
\end{split}
\end{equation}

The identities appearing in the first line of \eqref{mainxS_boundaries} model the peculiarity of the central 
node of 
the network to act zero resistence on the 
crossing particle (see \cite{209,204}) and they are encoded by Neumann-Kirchhoff 
boundary conditions. The last identities are the 
classical Dirichlet boundary conditions. Let 
$$H^s:=\prod_{k=1}^3H^s(e_k,\C)$$ for $s> 0$. We denote by $-\Delta$ the 
Laplacian operator appearing in \eqref{mainxS}-\eqref{mainxS_boundaries}, {\it i.e.} 
$$D(-\Delta)=\Big\{\psi=(\psi^1,\psi^2,\psi^3)\in H^2\ :\ \psi\ \text{verifies the boundary conditions
 }\eqref{mainxS_boundaries}\Big\}.$$ 
%
%
%
%
%

\begin{defi}
	Let $N\in\N^*$. We denote by $\AL\LL(N)$ the set of elements $\{L_j\}_{j\leq N}\in (\R^+)^N$ so that
$\big\{1,L_1,...,L_N\big\}$ are {linearly independent} over $\Q$ and all the ratios $L_k/L_j$ are 
{algebraic 
irrational numbers}.
\end{defi}
The set $\AL\LL(N)$ with $N\in\N^*$ contains the uncountable set of $\{L_{j}\}_{j\leq 
N}\in(\R^+)^N$ such
that each $L_j$ can be written in the form $t\widetilde L_j$ where all the ratios $\widetilde L_j/\widetilde L_k$ 
are algebraic irrational numbers and $t$ is a transcendental number. For instance, $\{\pi\sqrt{2},\pi\sqrt{3}\}$ 
belongs to $\AL\LL(2)$, while $\{\pi,\pi\sqrt{2},\pi\sqrt{3}\}$ 
to $\AL\LL(3)$.

\begin{teorema}\label{bim.1} 
Let $\mu:x\in (0,L_1)\longmapsto (x-L_1)^4$. There exists $\CC\subset (\R^+)^3$ countable such that, for every 
$\{L_1,L_2,L_3\}\in\AL\LL(3)\setminus \CC$, the dynamics of the system
\eqref{mainxS}-\eqref{mainxS_boundaries} is {globally exactly controllable} in
	$$H^{4+\epsilon}\cap D(|\Delta|^2),\ \ \ \ \ \ \ \ \epsilon>0.$$
In other words, let $\G_T^u$ be the unitary propagator associated to the dynamics of 
\eqref{mainxS}-\eqref{mainxS_boundaries} in the time interval $(0,T)$ with control $u\in L^2((0,T),\R)$. 
For every $\psi^1, \psi^2\in H^{4+\epsilon}\cap D(|\Delta|^2)$ such that 
$\|\psi^1\|_{L^2(\Gi,\C)}=\|\psi^2\|_{L^2(\Gi,\C)}$, 
$$\exists T>0,\ \ \ u\in L^2\big((0,T),\R\big)\ \ \ :\ \ \ \G_T^u\psi^1=\psi^2.$$
\end{teorema}

Theorem $\ref{bim.1}$ is proved in Section \ref{proofbim} and it is a consequence of the abstract result of
Theorem \ref{global}. There, the statement is guaranteed when specific hypotheses on the spectral behaviour 
of $A$ 
and on the field $B$ are satisfied.

 Theorem $\ref{bim.1}$ yields that the controllability of \eqref{mainxS}-\eqref{mainxS_boundaries} holds even 
though the external field 
only acts on the edge $e_1$. When the particle, represented by a state $\psi_0$, is 
(mostly) 
localized in $e_3$, it is 
still possible to move it to any other edge of the network by controlling the intensity of the field (Figure 
\ref{controllo_un_edge}).

\begin{figure}[H]
	\centering
	\includegraphics[width=\textwidth-100pt,height=70pt]{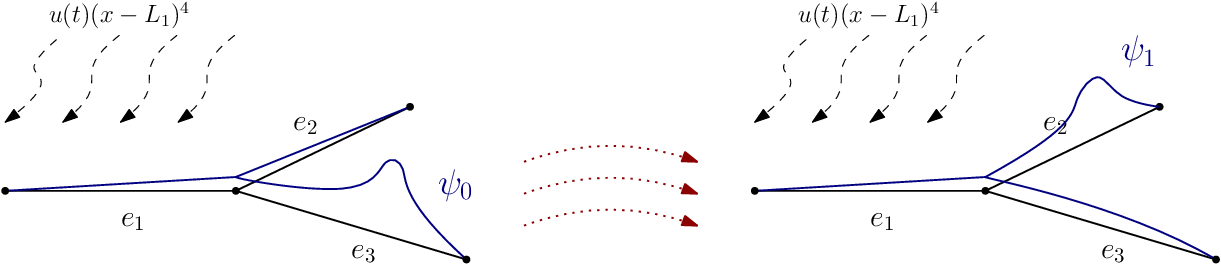}
	\caption{The figure represents an interesting application of Theorem \ref{bim.1}: it is possible to steer 
any state $\psi_0$, localized in $e_3$,  to any other state $\psi_1$, localized in $e_2$, by means of the dynamics 
of the 
bilinear 
Schr\"odinger equation \eqref{mainxS}-\eqref{mainxS_boundaries}, even though the potential field 
$u(t)\mu(x)=u(t)(x-L_1)^4$ only acts on the edge $e_1$.}\label{controllo_un_edge}
\end{figure}

\noindent
This peculiarity follows from the choice of $\{L_1,L_2,L_3\}\in\AL\LL(3)$. In this case, each bounded state 
of 
$i\dd_t\psi=-\Delta\psi$ is supported by the whole network (see Remark \ref{pallosso}) and then, it is
affected by the field acting on $e_1$. As a 
consequence, when we see $\psi_0$ as 
a superposition of 
bounded states, we realize that the control field ``see" the particle even 
though it is localized on a 
different edge from $e_1$.

\begin{osss} If $L_2/ 
L_3$ is
rational, then there exist 
eigenfunctions of $-\Delta$ vanishing in $e_1$. When $L_1/L_3$ is also 
rational, the spectrum of $-\Delta$ presents multiple eigenvalues (see Remark \ref{pallosso}). These are
obvious obstructions to the controllability of \eqref{mainxS}-\eqref{mainxS_boundaries} since the potential 
only acts on $e_1$. As a 
consequence, when we change the lengths of the branches of the network, it may happen that the 
controllability of \eqref{mainxS}-\eqref{mainxS_boundaries} is lost, even though the set of the "uncontrollable" 
lengths is 
countable.\end{osss}

\vspace{4mm}

\needspace{4mm}
{\noindent \bf \underline{Scheme of the work}}

\vspace{2mm}

\noindent
In Section $\ref{Sec2}$, we present the mathematical framework and the notations adopted in the work. 
In 
Section \ref{Sec3}, we state with Theorem \ref{global} our abstract global exact controllability result which we 
use to prove Theorem \ref{bim.1}. In the final part of this section, we also deal with a specific problem 
involving a tadpole graph.
In Section \ref{well}, we present some interpolation properties for the domains $D(|A|^\frac{s}{2})$ with $s>0$ 
and we ensure the well-posedness of the \eqref{mainx1}. 
In Section \ref{proofglobal}, we prove the abstract result of Theorem \ref{global} by extending a local 
exact controllability.
In {Appendix $\ref{numeri}$}, we ensure some spectral proprieties adopted in the manuscript. 
In {Appendix $\ref{momenti}$}, we provide a new technique leading to the solvability of the so-called {moment 
problem} appearing in the proof of the local exact controllability.
In {Appendix $\ref{analitics}$}, we develop some perturbation theory techniques.

\section{Preliminaries}\label{Sec2}
\subsection{Mathematical framework and notations}\label{notation}

Let $\Gi$ be a compact graph composed by $N\in\N^*$ edges $\{e_j\}_{j\leq N}$ of lengths $\{L_j\}_{j\leq N}$ and 
$M\in\N^*$ vertices $\{v_j\}_{j\leq M}$. For every vertex $v$, we denote 
\begin{align}\label{molteplicit}N(v):=\big\{l \in\{1,...,N\}\ |\ v\in 
e_l\big\},\ \ \  \ \ \ n(v):=|N(v)|.\end{align}
We call $V_e$ and $V_i$ the external and the internal vertices of 
the graph $\Gi$ (see Figure \ref{vertici}).
\begin{figure}[H]
	\centering
	\includegraphics[width=\textwidth-200pt, height=35pt]{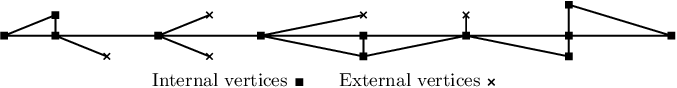}
	\caption{Internal and external vertices in a compact graph.}\label{vertici}
\end{figure}

We study graphs equipped with a metric, which parametrizes each edge $e_j$ with a coordinate going from $0$ to its 
length $L_j$. A graph is compact when it is composed by a finite number of vertices and edges of finite lengths. 
We 
consider functions $f:=(f^1,...,f^N):\Gi\rightarrow \C$ with domain a compact metric graph $\Gi$ so that 
$f^j:e_j\rightarrow \C$ for every $j\leq N$. We denote
$$L^2(\Gi,\C)=\prod_{j\leq N}L^2(e_j,\C).$$
The Hilbert space $L^2(\Gi,\C)$ is equipped with the norm $\|\cdot\|_{L^2}$ induced by the scalar product
$$\la\psi,\ffi\ra_{L^2}:=\sum_{j\leq N}\la\psi^j,\ffi^j\ra_{L^2(e_j,\C)}=\sum_{j\leq 
N}\int_{e_j}\overline{\psi^j}(x)\ffi^j(x)dx,\ \ \ \  \ \ \forall \psi,\ffi\in L^2(\Gi,\C).$$
For $s>0$, we define the spaces $$H^s=H^s(\Gi,\C):=\prod_{j=1}^N 
H^s(e_j,\C),\ \ \ \ \ \ \  h^s=\Big\{(x_j)_{j\in\N^*}\subset{\C}\ 
\big|\ \sum_{j=1}^{\infty}|j^s x_j|^2<\infty\Big\}.$$
We equip $h^s$ with the norm 
$\big\|(x_j)_{j\in\N^*}\big\|_{(s)}=\big(\sum_{j=1}^\infty|j^s x_j|^2\big)^\frac{1}{2}$ for every 
$(x_j)_{j\in\N^*}\in h^s$. Let $f\in H^1$ and $v$ be a vertex of $\Gi$ connected once to an edge $e_j$ with 
$j\leq N$. When the coordinate parametrizing $e_j$ in the vertex $v$ is equal to $0$ (resp. $L_j$), we denote
\begin{align}\label{NK1}\dd_x f^j(v)=\dd_xf^j(0),\ \ \  \ \ \ \ \ \big(\text{resp.}\ \dd_x 
f^j(v)=-\dd_xf^j(L_j)\big).\end{align}
When $e_j$ is a loop and it is connected to $v$ in both of its extremes, we use the notation 
\begin{align}\label{NK2}\dd_x 
f^j(v)=\dd_xf^j(0)-\dd_xf^j(L_j).\end{align} When $v$ is an external vertex and then $e_j$ is the only edge 
connected to $v$, we call $\dd_x f(v)=\dd_x f^j(v).$

In the bilinear Schr\"odinger equation $(\ref{mainx1})$, we consider the Laplacian $A$ being self-adjoint and we 
denote $\Gi$ as {quantum graph}. From now on, when we introduce a quantum graph $\Gi$, we implicitly define on 
$\Gi$ a self-adjoint Laplacian $A$. Formally, $D(A)$ is characterized by the following boundary conditions.

\needspace{3\baselineskip}
\begin{boundaries}
Let $\Gi$ be a compact quantum graph. 
\begin{itemize}
\item[($\Di$)] A vertex $v\in V_e$ is equipped with {Dirichlet} boundary conditions when $f(v)=0$ for every $f\in 
D(A)$.

\item[($\NN$)] A vertex $v\in V_e$ is equipped with {Neumann} boundary conditions when $\dd_xf(v)=0$ for every 
$f\in D(A)$.
 \item[($\NN\KK$)] A vertex $v\in V_i$ is equipped with {Neumann-Kirchhoff} boundary conditions when every $f\in 
D(A)$ is 
continuous in $v$ and $\sum_{j\in N(v)}\dd_x f^j(v)=0$ (according to the notations \eqref{NK1} and \eqref{NK2}).

\end{itemize}
\end{boundaries}

\needspace{1\baselineskip}
\begin{notations}Let $\Gi$ be a compact quantum graph. 
\begin{itemize}

\item The graph $\Gi$ is said to be equipped with ($\Di$) when the Laplacian $A=-\Delta$ in $L^2(\Gi,\C)$ 
is such that
$$D(A)=\{\psi\in H^2\ :\ \psi\text{ satisfies $(\NN\KK)$ in every $v\in V_i$ and ($\Di$) in 
each $v\in V_e$}\}.$$

\item The graph $\Gi$ is said to be equipped with ($\NN$) when the Laplacian $A=-\Delta$ in $L^2(\Gi,\C)$  is 
such that
$$D(A)=\{\psi\in H^2\ :\ \psi\text{ satisfies $(\NN\KK)$ in every $v\in V_i$ and ($\NN$) in 
each $v\in V_e$}\}.$$

\item The graph $\Gi$ is said to be equipped with ($\Di$/$\NN$) when the Laplacian $A=-\Delta$ in $L^2(\Gi,\C)$ is 
such that 
$$D(A)=\{\psi\in H^2\ :\ \psi\text{ satisfies $(\NN\KK)$ in every $v\in V_i$ and each $v\in V_e$ verifies ($\Di$) 
or 
($\NN$)}\}.$$
\end{itemize}
\end{notations}

\begin{osss}\label{compact_resolvent}
When the boundary conditions described above are satisfied, the Laplacian $A$ is self-adjoint (see \cite[Theorem\ 
3]{kuk}), it admits compact resolvent and then purely discrete spectrum (see \cite[Theorem\ 18]{kuk}).\end{osss}

We 
denote by 
$(\lambda_k)_{k\in\N^*}$ the ordered sequence of eigenvalues of $A$ and 
$(\phi_k)_{k\in\N^*}$ is a Hilbert basis of $L^2(\Gi,\C)$ composed by corresponding eigenfunctions. Let $[r]$ be 
the entire part of a number $r\in\R$. For $s>0$, we denote
\begin{equation*}
\begin{split}
H^s_{\NN\KK}:=\Big\{&\psi\in H^s(\Gi,\C)\ \Big|\ \dd_x^{2n}\psi\text{ is continuous in }v,\ 
 \forall n\in\N,\ n<\big[({s+1})/{2}\big],\ \forall v\in V_i;\\
& \sum_{j\in N(v)}\dd_{x}^{2n+1}\psi^j(v)=0,\ \forall n\in\N,\ n<\big[{s}/{2}\big],\  \forall v\in V_i\Big\},\\
H^s_{\Gi}&:=D(A^{{s}/{2}}),\ \ \ \ \ \ \ \ \ \ \ \ \ \ 
\|\cdot\|_{(s)}:=\|\cdot\|_{H^s_{\Gi}}=\Big(\sum_{k\in\N^*}\big|k^s\la\cdot,\phi_k\ra_{L^2}\big|^2\Big)^{\frac{1}{2
} } .\\
\end{split}
\end{equation*}

\subsection{Spectral properties}

The following proposition rephrases the results of $\cite[Theorem  \ 3.1.8]{bomber}$ and 
$\cite[Theorem  \ 3.1.10]{bomber}$. There, we denote 
by $\big(\lambda_k^{\widehat\Gi}\big)_{k\in\N^*}$ the ordered sequence of eigenvalues of 
$A$ on a compact quantum graph $\widehat\Gi$. 

\begin{prop}{$\cite[Theorem  \ 3.1.8]{bomber}$ \& $\cite[Theorem  \ 3.1.10]{bomber}$}\label{spazio} Let $\Gi$ be a 
compact quantum graph.
\smallskip

\noindent
{\bf 1)} Let $w$ be a vertex of $\Gi$. If $\Gi^{\Di}$is the graph obtained by 
changing the boundary conditions in $w$ with ($\Di$), then $$\lambda_k^{\Gi}\leq\lambda_k^{\Gi^{\Di}}\leq 
\lambda_{k+1}^{\Gi},\ \ \ \ \ \ \ \ k\in\N^*.$$

\noindent
{\bf 2)} Let $w$ and $v$  be two vertices of $\Gi$ equipped with ($\NN\KK$) or ($\NN$). If $\Gi '$ is the 
graph obtained by merging the 
vertices $w$ and $v$ of $\Gi$ in one unique vertex equipped with ($\NN\KK$), then 
$$\lambda_k^{\Gi}\leq\lambda_k^{\Gi 
'}\leq 
\lambda_{k+1}^{\Gi},\ \ \ \ \ \ \forall k\in\N^*.$$
\end{prop}

Let $\Gi$ a be compact graphs composed by $N\in\N^*$ edges. We define $\Gi^\Di$ the quantum graph obtained by 
imposing $(\Di)$ in each vertex of $\Gi$: $\Gi^\Di$ consists in $N$ disjoint intervals equipped with $(\Di)$. Let 
$\Gi^\NN$ be constructed from 
$\Gi$ by disconnecting each edge and by imposing $(\NN)$ in each vertex: $\Gi^\NN$ consists in $N$ disjoint 
intervals equipped with $(\NN)$. From Proposition $\ref{spazio}$, 
\begin{align}\label{spezzo}\lambda_{k-2N}^{\Gi^\NN}\leq\lambda_k^{\Gi}\leq\lambda_{k+M}^{\Gi^{\Di}},\ \ \ \ \ 
\forall k>2N.\end{align} The sequences 
$\lambda_k^{\Gi^\NN}$ and $\lambda_k^{\Gi^\Di}$ are respectively obtained by reordering 
$\big\{\frac{k^2\pi^2}{L_l^2}\big\}_{\overset{k\in\N^*}{l\leq N}}$ and 
$\big\{\frac{(k-1)^2\pi^2}{L_i^2}\big\}_{\overset{k\in\N^*}{i\leq N}}.$ 
	Now,
$$	\lambda_{l-2N}^{\Gi^\NN}\geq   \frac{(l-2N-1)^2\pi^2}{N^2\tilde m}\geq \frac{l^2\pi^2}{2^{2(2N+1)}N^2\tilde 
m},\ \ \ \ \ \  \ \ \ \lambda_{l+M}^{\Gi^\Di}\leq \frac{(l+M)^2\pi^2}{\hat m}\leq \frac{l^2 
2^{2M}\pi^2}{\hat 
m}$$
for $\tilde m=\max_{j\leq N}L_j^2$ and $\hat m=\min_{j\leq N}L_j^2$. Finally, the last two relations 
and \eqref{spezzo} lead 
to the following lemma.

\begin{lemma}\label{interessante}
Let $(\lambda_k)_{k\in\N*}$ be the eigenvalues of a self-adjoint Laplacian $A$ defined on a compact quantum 
graph 
equipped with ($\Di$), ($\NN$) or ($\Di/\NN$). There exist 
$C_1,C_2>0$ such that \begin{align*}C_1 k^2\leq \lambda_k\leq C_2 k^2,\ \ \ \ \ \ \ 
\forall k\geq 2.\end{align*}
\end{lemma}
When a compact quantum graph $\Gi$ is not equipped with $(\NN)$, we have
$0\not\in\sigma(A)$ (the spectrum of 
$A$) and, from Lemma \ref{interessante}, there holds $\|\cdot\|_{(s)}\asymp \||A|^\frac{s}{2}\cdot\|_{L^2}$, {\it 
i.e.} 
$$\exists C_1,C_2>0 \ \ \ :\ \ \ C_1\|\cdot\|_{(s)}\leq\||A|^\frac{s}{2}\cdot\|_{L^2}\leq C_2\|\cdot\|_{(s)}.$$ 
When $\Gi$ is equipped with $(\NN)$	, we have $0\in\sigma(A)$ and $\lambda_1=0$. There exists 
$c\in\R$ such that $0\not\in\sigma(A+c)$ and 
$$\|\cdot\|_{(s)}\asymp \||A+c|^\frac{s}{2}\cdot\|_{L^2}.$$
Now, from $\cite[Proposition\ 6.2]{wave}$, there exist $\MM\in\N^*$ and $\delta'>0$ such that 
$\inf_{{k\in\N^*}}|\sqrt\lambda_{k+\MM}-\sqrt\lambda_k|>\delta' \MM$ and
\begin{equation}\begin{split}\label{g13_aus}
\inf_{{k\in\N^*}}|\lambda_{k+\MM}-\lambda_k|\geq 
\sqrt{\lambda_{\MM+1}}\inf_{{k\in\N^*}}|\sqrt\lambda_{k+\MM}-\sqrt\lambda_k|\geq\delta\MM\\
\end{split}\end{equation}
for $\delta=\sqrt{\lambda}_{\MM+1}\delta'$. The relation \eqref{g13_aus} yields the following lemma.

\begin{lemma}\label{g13}
Let $(\lambda_k)_{k\in\N*}$ be the eigenvalues of a self-adjoint Laplacian $A$ defined on a compact quantum 
graph 
equipped with ($\Di$), ($\NN$) or ($\Di/\NN$). There exist $\delta>0$ and
$\MM\in\N^*$ such that 
\begin{equation*}\begin{split}
\inf_{{k\in\N^*}}|\lambda_{k+\MM}-\lambda_k|\geq \delta\MM.\\
\end{split}\end{equation*}
\end{lemma}

In the next proposition, we use Proposition \ref{spazio} in order to characterize the asymptotic behaviour of 
$(\lambda_k)_{k\in\N^*}$ when $\Gi$ is one of the graphs represented in Figure $\ref{int}$.

\begin{prop}\label{final_spectral} Let $\Gi$ be either a tadpole, a 
two-tails tadpole, a double-rings graph or a star graph with $N\leq 4$ edges. Let $\Gi$ be equipped with 
($\Di$/$\NN$). If $\{L_j\}_{j\leq N}\in\AL\LL(N)$, then, for every $\epsilon>0$, there exists $C>0$ such that
\begin{equation*}|\lambda_{k+1}-\lambda_k|\geq C  k^{-\epsilon},\ \ \ \ \ \ \ \ \forall k\in\N^*\end{equation*}
	\end{prop}
	
	Before providing the proof of Proposition \ref{final_spectral}, we introduce the following auxiliary 
result.

\begin{lemma}\label{figo}
Let $\{L_l\}_{l\leq N_1},\{\tilde L_i\}_{i\leq N_2}\subset \R$ with $N_1,N_2\in\N^*$.
	Let $\big(\lambda_k^1\big)_{k\in\N^*}$ and $\big(\lambda_k^2\big)_{k\in\N^*}$ be the two sequences of numbers 
obtained by reordering 
$\big\{\frac{k^2\pi^2}{L_l^2}\big\}_{\underset{l\leq N_1}{k,l\in\N^*}}$ and $\big\{\frac{k^2\pi^2}{\tilde 
L_i^2}\big\}_{\underset{i\leq N_2}{k,i\in\N^*}}$ respectively.
	When all the ratios ${\tilde L_i}/{L_l}$ are algebraic irrational numbers, for every $\epsilon>0$, there 
exists $C>0$ such that$$ |\lambda_{k+1}^1-\lambda_k^2|\geq{C}{k^{-\epsilon}},\ \ \ \ \ \forall k\in\N^*.$$
\end{lemma}
\begin{proof}
 See Appendix \ref{numeri}.\qedhere
\end{proof}

\begin{proof}[Proof of Proposition \ref{final_spectral}] Let $\Gi$ be a {tadpole graph} equipped with ($\Di$). We 
construct from $\Gi$ two quantum graphs $\Gi^\NN$ and $\Gi^\Di$ as follows (see the first line of Figure 
$\ref{decomposizione}$ for further details). Let 
$\Gi^{\NN}$ be the graph obtained by disconnecting the edge $e_1$, representing the ``head" of the tadpole, on one 
side. We impose ($\NN$) on the new external 
vertex of $e_1$ created by this procedure. Let $\Gi^{\Di}$ be 
obtained from $\Gi$ by imposing ($\Di$) on its internal vertex 
$v\in V_i$. 
We respectively denote by $\big(\lambda_k^{\Gi^{\Di}}\big)_{k\in\N^*}$ and 
$\big(\lambda_k^{\Gi^{\NN}}\big)_{k\in\N^*}$ 
the ordered sequences of eigenvalues in $\Gi^\Di$ and $\Gi^\NN$ which are obtained by reordering 
$\big\{\frac{k^2\pi^2}{L_j^2}\big\}_{\underset{j\in\{1,2\}}{k\in\N^*}}$ and $\big\{\frac{(2k-1)^2\pi^2}{4 
(L_1+L_2)^2}\big\}_{{k\in\N^*}}$. From Proposition $\ref{spazio}$, 
\begin{equation}\label{why}\lambda_k^{\Gi}\leq\lambda_k^{\Gi^{\Di}}\leq \lambda_{k+1}^{\Gi},\ \ \ \ \ \ \ \ \ \ \ 
\ \ \ \ \lambda_k^{\Gi}\leq\lambda_{k+1}^{\Gi^{\NN}}\leq \lambda_{k+1}^{\Gi},\ \ \ \ \ \ \  \ \ \ \ \ \forall 
k\in\N^*.\end{equation}
	If $\{L_1,L_2\}\in\AL\LL(2)$, then the ratios between the elements in $\{L_1,L_2, L_1+L_2\}$ are 
algebraic irrational numbers. Lemma $\ref{figo}$ ensures the 
existence of $C>0$ such that, for every $\epsilon>0$, there holds 
$$|\lambda_{k+1}^{\Gi}-\lambda_k^{\Gi}|\geq|\lambda_{k+1}^{\Gi^{\NN}}-\lambda_k^{\Gi^{\Di}}|\geq{Ck^{-\epsilon}},\ 
\ \ \ \ \ \forall k\in\N^*.$$
	The claim is guaranteed when $\Gi$ is a tadpole graph.
	 The same techniques are also valid when $\Gi$ is a {tadpole graph} equipped with ($\NN$), when $\Gi$ is a 
{two-tails tadpole graph}, a {double-rings graph} or a {star graph} with $N\leq 4$ edges. In every 
framework, we impose that $\{L_k\}_{k\leq N}\in\AL\LL(N)$. In Figure $\ref{decomposizione}$, we represent how to 
define $\Gi^\NN$ and $\Gi^{\Di}$ from the corresponding $\Gi$. \qedhere\end{proof}

\begin{figure}[H]
	\centering
	\includegraphics[width=\textwidth-50pt,height=135pt]{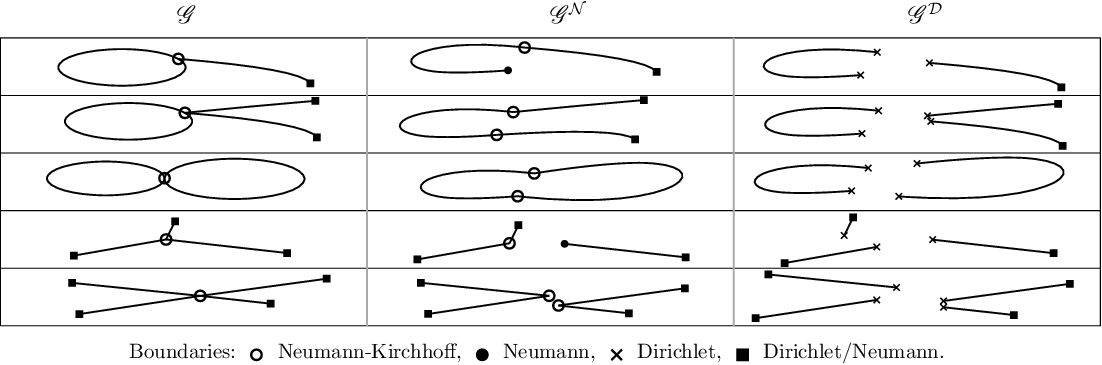}
	\caption{The figure represents the graphs described in the proof of Proposition $\ref{final_spectral}$. The 
first 
column  
shows the graphs $\Gi$ considered: tadpole, two-tails tadpole, double-rings graph, star graph with $N=3$ and star 
graph with $N=4$. The second and the third columns respectively provide the corresponding 
graphs $\Gi^{N}$ and 
$\Gi^{D}$.}\label{decomposizione}
\end{figure}

\section{Main results}\label{Sec3}
\subsection{Abstract controllability result}
Let $\eta>0$ and $a\geq 0$. We denote by $I$ the subset of 
$(\N^*)^2$ such that $I:=\{(j,k)\in(\N^*)^2:j< k\}$.

\needspace{3\baselineskip}
\begin{assumptionI}[$\eta$] 
	The bounded symmetric operator $B$ satisfies the following conditions.
	\noindent
	\begin{enumerate}
		\item There exists $C>0$ such that $|\la\phi_j,B\phi_1\ra_{L^2}|\geq\frac{C}{j^{2+\eta}}$ for every 
$j\in\N^*$.
		\item For $(j,k),(l,m)\in I$ such that $(j,k)\neq(l,m)$ and such that
		$\lambda_j-\lambda_k=\lambda_l-\lambda_m,$ we have
		$$\la\phi_j,B\phi_j\ra_{L^2}-\la\phi_k,B\phi_k\ra_{L^2}\neq 
\la\phi_l,B\phi_l\ra_{L^2}-\la\phi_m,B\phi_m\ra_{L^2}.$$
	\end{enumerate}
\end{assumptionI}

The first condition of Assumptions I quantifies how much $B$ mixes 
the eigenspaces associated to the eigenfunctions $(\phi_k)_{k\in\N^*}$. This 
assumption is crucial for the controllability. Indeed, when $B$ stabilizes such spaces, also $\G_T^u$, the unitary 
propagator associated to the \eqref{mainx1}, does the 
same 
and we can not expect to obtain controllability results. The second hypothesis is used to decouple some 
eigenvalues 
resonances appearing in the proof of the approximate controllability that we use in order to prove our main 
results. 

\needspace{3\baselineskip}
\begin{assumptionII}[$\eta,a$]
	Let $Ran(B|_{H^2_{\Gi}})\subseteq H^2_{\Gi}$ and one of the following assumptions be satisfied.
	\begin{enumerate}

		\item When $\Gi$ is equipped with ($\Di$/$\NN$) and $a+\eta\in(0, 3/2)$, there exists 
		$d\in[\max\{a+\eta,1\},3/2)$ such that $Ran(B|_{H_{\Gi}^{2+d}})\subseteq H^{2+d}\cap H^2_{\Gi}.$

		\item When $\Gi$ is equipped with ($\NN$) and $a+\eta\in(0, 7/2)$, there exist $d\in[\max\{a+\eta,2\},7/2)$ 
and $d_1\in(d,7/2)$ such that $Ran(B|_{H_{\Gi}^{2+d}})\subseteq H^{2+d}\cap H^{1+d}_{\NN\KK}\cap H^2_{\Gi}$ and 
$Ran(B|_{ H^{d_1}_{\NN\KK}})\subseteq H^{d_1}_{\NN\KK}.$

		\item When $\Gi$ is equipped with ($\Di$) and $a+\eta\in(0, 5/2)$, there exists 
$d\in[\max\{a+\eta,1\},5/2)$ such that $Ran(B|_{H_{\Gi}^{2+d}})\subseteq H^{2+d}\cap H^{1+d}_{\NN\KK}\cap 
H^2_{\Gi}.$ If $d\geq 2$, then there exists $d_1\in(d,5/2)$ such that $Ran(B|_{H^{d_1}})\subseteq H^{d_1}.$

	\end{enumerate}
\end{assumptionII}
Assumptions II calibrate the regularity of the control potential $B$ according to the choice of the 
boundary conditions defining $D(A)$ which affects the definition of the spaces $H^s_\Gi=D(|A|^\frac{s}{2})$ with 
$s>0$.

We are finally ready to present our main abstract controllability result for the \eqref{mainx1} on general 
networks. 

\begin{defi}
Let $\G_T^u$ be the unitary propagator associated to \eqref{mainx1} with $T>0$ and $u\in L^2((0,T),\R)$.
	The $(\ref{mainx1})$ is said to be globally exactly controllable in $H^s_\Gi$ with $s>0$ when, for every 
$\psi^1,\psi^2\in H^s_\Gi$ such that $\|\psi^1\|_{L^2}=\|\psi^2\|_{L^2}$, there exist $T>0$ and $u\in 
L^2((0,T),\R)$ such that $\G_T^u\psi^1=\psi^2.$
\end{defi}

\begin{teorema}\label{global} 
Let $\Gi$ be a compact quantum graph. Let $\MM\in\N^*$ be defined in Lemma \ref{g13}. Let exist $C>0$ and 
$\tilde d\geq 0$ such that
\begin{equation}\label{gapp}|\lambda_{k+1}-\lambda_k|\geq C  k^{-\frac{\tilde d}{\MM-1}},\ \ \ \ \ \ \ \ \forall 
k\in\N^*.\end{equation}
If the couple 
$(A,B)$ satisfies Assumptions I$(\eta)$ and Assumptions II$(\eta,\tilde d)$ for some $\eta>0$, then the 
$(\ref{mainx1})$ is globally exactly controllable in $H^{s}_{\Gi}$ for $s=2+d$ and $d$ from Assumptions II.
\end{teorema}
\begin{proof}
	See Section $\ref{proofglobal}$.\qedhere
\end{proof}

%

In the next proposition, we state an abstract global exact controllability result valid when $\Gi$ 
is one 
of the graphs represented in Figure $\ref{int}$. This 
result leads to Theorem $\ref{bim.1}$.

\begin{prop}\label{final} Let $\{L_j\}_{j\leq N}\in\AL\LL(N)$. Let $\Gi$ be either a tadpole, a two-tails 
tadpole, a double-rings graph or a star graph with $N\leq 4$ edges. Let $\Gi$ be equipped with ($\Di$/$\NN$). If 
the couple $(A,B)$ satisfies Assumptions I$(\eta)$ and Assumptions II$(\eta,\epsilon)$ for some $\eta,\epsilon>0$, 
then the $(\ref{mainx1})$ is globally exactly controllable in $H^{s}_{\Gi}$ for $s=2+d$ and $d$ from Assumptions 
II.
\end{prop}
	\begin{proof}
	 The claim follows by the validity of the spectral hypothesis of Theorem $\ref{global}$ due to 
Proposition \ref{final_spectral}.\qedhere
	\end{proof}

\begin{osss}\label{finalcoro}
	Let $\{L_j\}_{j\leq 2}\in\AL\LL(2)$. Proposition \ref{final_spectral} and then Proposition $\ref{final}$ are 
also 
valid when $\Gi$ is a two-tails tadpole with 
tails of length $L_2$ and head $L_1$. The same is true when $\Gi$ is a star graph 
with $3$ (or $4$) edges such that two edges are long $L_1$ and the remaining one (resp. ones) $L_2$.
\end{osss}

%
%
%
%
%

The size of the time in Theorem \ref{global} depends on the 
initial 
and the final states of the dynamics. This is due to the global approximate controllability result adopted in the 
proof of Theorem \ref{global}. Nevertheless, the local exact controllability (presented in Proposition \ref{lcl}), 
is valid for any $T>0$ (see 
Remark \ref{time} for further details).

\subsection{Proof of Theorem \ref{bim.1}}\label{proofbim}
\begin{proof}
	Let $\Gi$ be a star graph with $3$ edges of lengths $\{L_j\}_{j\leq 3}$ equipped ($\Di$). The ($\Di$) 
conditions on the external vertices imply that each eigenfunction $\phi_j$ with $j\in\N^*$ satisfies 
$\phi_j^l(0)=0$ for every $l\leq 3$. Then, 
$$\phi_j(x)=\big(a^1_j\sin(x\sqrt{\lambda}_j),a_j^2\sin(x\sqrt{\lambda}_j),a_j^3\sin(x\sqrt{\lambda}_j)\big)$$
	with $\{a_j^l\}_{l\leq 3}\subset\C$ such that $(\phi_j)_{j\in\N^*}$ forms a Hilbert basis of $L^2(\Gi,\C)$, 
{\it i.e.} \begin{align}\label{tremor1}\sum_{l\leq 
3}\int_{0}^{L_l}|a_j^l|^2\sin^2(x\sqrt{\lambda}_j)dx=\sum_{l\leq 
3}|a_j^l|^2\Big(\frac{L_l}{2}+\frac{\cos(L_l\sqrt{\lambda}_j)\sin(L_l\sqrt{\lambda}_j)}{2\sqrt{\lambda}_j}
\Big)=1.\end{align}
For every $j\in\N^*$, the ($\NN\KK$) condition in $V_i$ yields
	 \begin{equation}\label{tremor}\begin{split}
	 a^1_j\sin(\sqrt{\lambda}_jL_1)=...=a^3_j\sin(\sqrt{\lambda}_jL_3),\ \ \ \ \ \ 
\sum_{l\leq 3} 
a^l_j\cos(\sqrt{\lambda}_jL_l)=0,\\
	 \sum_{l\leq 3} \cot(\sqrt{\lambda}_jL_l)=0,\ \ \ \ \ \ \ \ \ \  \sum_{l\leq 
3}|a_j^l|^2{\sin(L_l\sqrt{\lambda}_j)\cos(L_l\sqrt{\lambda}_j)}=0.\\ 
	 \end{split}
	 \end{equation} 
	Now, \eqref{tremor1} and \eqref{tremor} ensure $1=\sum_{l=1}^3|a_j^l|^2{L_l}/{2}$. The continuity implies 
$a_j^l=a_j^1\frac{\sin(\sqrt{\lambda}_j 
L_1)}{\sin(\sqrt{\lambda}_j L_l)}$ for $l\neq 1$ and 
	\begin{equation}\label{pistola}
	\begin{split}
|a_j^1|^2\Bigg(L_1+\sum_{l=2}^3L_l\frac{\sin^2(\sqrt{\lambda}_j L_1)}{\sin^2(\sqrt{\lambda}_j L_l)}\Bigg)={2},\ \ 
\  \ \ \ \ \Longrightarrow\ \ \ \ \ \ |a_j^1|^2&=\frac{2\prod_{m\neq 1 }\sin^2(\sqrt{\lambda}_j L_m)}{\sum_{k=1}^3 
L_k\prod_{m\neq k}\sin^2(\sqrt{\lambda}_j L_m)}.
	\end{split}
	\end{equation}
	From $(\ref{tremor})$ and $(\ref{pistola})$, we have $\sum_{l=1}^3\cos(\sqrt{\lambda}_k L_l)\prod_{m\neq l 
}\sin(\sqrt{\lambda}_k L_m)=0$. The validity of $\cite[Proposition\ A.11]{wave}$ and Lemma   
\ref{interessante} ensure that, for every $\epsilon>0$, there exist $C_1,C_2>0$ such that, for every $j\in\N^*$,
	\begin{equation}
	\label{oip}
	\begin{split}
	|a_j^1|&=\sqrt{\frac{2}{\sum_{l=1}^3L_l\sin^{-2}(\sqrt{\lambda}_j 
L_l)}}\geq\sqrt{\frac{2}{\sum_{l=1}^3L_lC_1^{-2}{\lambda_j^{1+\epsilon}}}}\geq \frac{C_2}{j^{1+\epsilon}} .\\
	\end{split}
	\end{equation}
	{\bf 1) Validation of Assumptions I.1\ .} We notice $\la\phi_k^l, (B\phi_j)^l\ra_{L^2(e_j,\C)}=0$ for $l\neq 
1$ and $k,j\in\N^*$. Let
	$$
	a_j(x):=\frac{2\prod_{m\neq 1 }\sin^2(\sqrt{\lambda_j} L_m)}{\sum_{k=2}^4L_k\sin^2(\sqrt{\lambda_j} 
x)\prod_{{m\neq k, 1}}\sin^2(\sqrt{\lambda_j} L_m)+x\prod_{{m\neq 1}}\sin^2(\sqrt{\lambda_j} L_m)},$$
$$B_1(x):=\frac{-30 \sqrt{\lambda_1} x+20 \sqrt{\lambda_1}^3 x^3+4 \sqrt{\lambda_1}^5 x^5+15 \sin(2 
\sqrt{\lambda_1} x)}{40 \sqrt{\lambda_1}^5},$$
	\begin{equation*}
	\begin{split}B_j(x)&:=2\frac{ -6(\sqrt{\lambda_1}-\sqrt{\lambda_j}) x + (\sqrt{\lambda_1}-\sqrt{\lambda_j})^3 
x^3 + 6 \sin((\sqrt{\lambda_1}-\sqrt{\lambda_j}) x)}{(\sqrt{\lambda_1}-\sqrt{\lambda_j})^5}\\
	&-2\frac{ -6(\sqrt{\lambda_1}+\sqrt{\lambda_j}) x + (\sqrt{\lambda_1}+\sqrt{\lambda_j})^3 x^3 + 6 
\sin((\sqrt{\lambda_1}+\sqrt{\lambda_j}) x)}{(\sqrt{\lambda_1}+\sqrt{\lambda_j})^5}\\
	\end{split}\end{equation*}
with $j\in\N^*$. Each function $\widetilde B_j(\cdot):=\sqrt{a_1(\cdot)}\sqrt{a_j(\cdot)}B_j(\cdot)$ is 
non-constant and analytic in $\R^+$, while we notice that $B_{1,j}=\la\phi_1, B\phi_j\ra_{L^2}=\widetilde B_j(L_1)$ 
by calculation. The set of positive zeros $\tilde V_j$ of each $\widetilde B_j$ is a discrete subset of $\R^+$ and 
$\tilde V=\bigcup_{j\in\N^*}\tilde V_j$ is countable. For every $\{L_l\}_{l\leq 4}\in \AL\LL(4)$ such that 
$L_1\not\in \tilde V$, we have $|B_{1,j}|\neq 0$ for every $j\in\N^*$. Now, there holds $|B_{1,j}|\sim 
{|a_j|L_1\sqrt{\lambda_1}\sqrt{\lambda_j}}{(\lambda_j-\lambda_1)^{-2}}$ for every $j\in\N^*\setminus\{1\}.$	From 
Lemma $\ref{interessante}$ and the identity $(\ref{oip})$, the first point of Assumptions I($2+\epsilon$) is 
verified as, for each $\epsilon>0$, there exists $C_3>0$ such that $$|B_{1,j}|\geq 
\frac{C_3}{j^{4+\epsilon}},\ \ \ \ \ \forall j\in\N^*.$$

\noindent
{\bf 2) Validation of Assumptions I.2\ .} By calculation, we notice that $B_{j,j}=\la\phi_j, 
B\phi_j\ra_{L^2}=F_j(L_1)$ where 
\begin{equation*}
	\begin{split}
F_j(x)&:=a_j(x)\frac{-30 \sqrt{\lambda}_k x+20 \sqrt{\lambda}_k^3 x^3+4 \sqrt{\lambda}_k^5 x^5+15 \sin(2 
\sqrt{\lambda}_k x)}{40 \sqrt{\lambda}_k^5}.\\
	\end{split}
\end{equation*}	
 Let $(k,j),(m,n)\in I,\ (k,j)\neq(m,n)$ for 
$I$ defined above Assumptions I. For 
$F_{j,k,l,m}(x)=F_j(x)-F_k(x)-F_l(x)+F_m(x),$ it follows $F_{j,k,l,m}(L_1)=B_{j,j}-B_{k,k}-B_{l,l}+B_{m,m}$ and 
$F_{j,k,l,m}(x)$ is a non-constant analytic function for $x>0$. Furthermore $V_{j,k,l,m}$, the set of the positive 
zeros of $F_{j,k,l,m}(x)$, is discrete and $V:=\bigcup_{\underset{j\neq k\neq l\neq 
m}{j,k,l,m\in\N^*}}V_{j,k,l,m}$ 
is a countable subset of $\R^+$. For each $\{L_l\}_{l\leq 3}\in \AL\LL(3)$ such that $L_1\not\in V\cup \tilde V$, 
Assumptions I$(2+\epsilon)$ are verified.
	
	\medskip
	
	\noindent
	{\bf 3) Validation of Assumptions II.3 and conclusion.} We notice that $B$ stabilizes the spaces $H^2_\Gi$, 
$H^m$ and $H^m_{\NN\KK}$ for $m\in (0,9/2)$ since, for $n\in\N^*$ such that $n<5$, we have
$$\dd_x^{n-1}(B\psi)^1(L_1)=....=\dd_x^{n-1}(B\psi)^3(L_3)=0,\ \ \ \ \ \ \forall \psi\in H^n_{\NN\KK},$$ which 
implies 
$B\psi\in H^n_{\NN\KK}.$ The third point of Assumptions II($2+\epsilon_1,\epsilon_2$) is valid 
for each $\epsilon_1,\epsilon_2>0$ such that $\epsilon_1+\epsilon_2\in(0,1/2)$. From Proposition $\ref{final}$, 
the 
controllability holds in $H_{\Gi}^{4+\epsilon}$ with $\epsilon>0$. Finally, we note that 
$H_{\Gi}^{4+\epsilon}=H^{4+\epsilon}\cap H^4_\Gi$ (see Proposition \ref{bor} for further details).\qedhere
\end{proof}

\begin{osss}\label{stella_neumann}
The proof Theorem \ref{bim.1} can also be adapted for star-graphs equipped with 
($\NN$) or with ($\Di/\NN$). In such cases, we respectively have to use Proposition \ref{zuazua} and Remark 
\ref{zuazuafigo} instead of $\cite[Proposition\ A.11]{wave}$.
\end{osss}

\begin{osss}\label{pallosso}
Let us consider a three edges star graph $\Gi$ equipped with ($\Di$). The same observations can be done for star 
graphs of $N\in\N^*$ edges and equipped with ($\Di$/$\NN$)

\smallskip

\noindent
{\bf 1)} When $\{L_1,L_2,L_3\}\in\AL\LL(3)$, each 
eigenfuction 
$\phi=(\phi^1,\phi^2,\phi^3)$ of $A$ is such that $\phi^j\not\equiv 0$ for every $j\leq 3$. Indeed, if for 
instance 
$\phi^1\equiv 0$, then the continuity would imply $\phi^2(L_2)=\phi^3(L_3)=0$ and then 
the corresponding eigenvalue $\lambda$ would be the form $\lambda=\frac{n_1^2\pi^2}{L_2^2}$ and 
$\lambda=\frac{n_2^2\pi^2}{L_3^2}$ for suitable $n_1,n_2\in\N^*$ which is impossible.

\smallskip

\noindent
{\bf 2)} When $L_2/ 
L_3$ is
rational, there exist $n,m\in\N^*$ such that ${L_2}/{L_3}={n}/{m}$. In this case, there exist
eigenfuctions of $A$ of the form $(0, \sin(\sqrt\mu x),-\sin(\sqrt\mu x) )$ where $\mu = \frac{k^2 
m^2\pi^2 }{L_3^2}$ is the corresponding eigenvalue for some $k\in\N^*$. 
When also $L_1/L_3$ is 
rational, there exist $n',m',\in\N^*$ such that ${L_1}/{L_3}={n'}/{m'}.$ The sequence $\{\mu_k\}_{k\in\N^*}$ 
with $\mu_k=\frac{k^2m^2 m'^2\pi^2}{L_3^2}$ is composed by eigenvalues of $A$ and they are 
multiple. Indeed, fixed $k\in\N^*$,
$$f_k=\big(-2\sin(\sqrt\mu_k x),\sin(\sqrt\mu_k x),\sin(\sqrt\mu_k x)\big),\ \ \ \ \ \ \ \ 
g_k=\big(0,\sin(\sqrt\mu_k 
x),-\sin(\sqrt\mu_k x)\big)$$
are reciprocally orthogonal eigenfunctions of $A$ corresponding to $\mu_k$.
\end{osss}

\subsection{Controllability of a bilinear quantum system on a tadpole graph}\label{tadpole}

Another application of Theorem \ref{global} is the following. Let $\Gi$ be a {tadpole graph} composed by two 
edges $\{e_1,e_2\}$ connected in an internal vertex $v$. The edge $e_1$ is self-closing and parametrized in the 
clockwise direction with a coordinate going from $0$ to $L_1$ (the length $e_1$). On the \virgolette{tail} $e_2$, 
we consider a coordinate going from $0$ in the to $L_2$ and we associate the $0$ to the external vertex $\tilde v$. 

\begin{figure}[H]
	\centering
	\includegraphics[width=\textwidth-130pt]{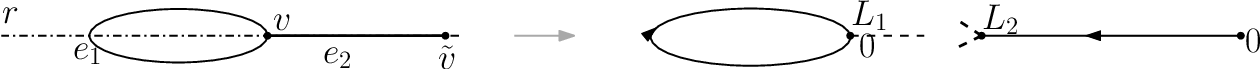}
	\caption{The parametrization of the tadpole graph and its symmetry axis $r$.}\label{parametrizzazione1}
\end{figure}

\begin{teorema}\label{bim.2} 
	Let $\Gi$ be a tadpole graph equipped with ($\Di$). Let $B:\psi\in L^2(\Gi,\C)\longrightarrow 
(\mu_1\psi^1,\mu_2\psi^2)$ with 
$$\mu_1(x):=\sin\big(\frac{2\pi}{L_1}x\big)+x(x-L_1),\ \ \ \  \ \mu_2(x):=x^2-(2L_1+2L_2)x+L_2^2+2L_1L_2.$$
There exists $\CC\subset (\R^+)^2$ countable so that, for each $\{L_1,L_2\}\in 
\AL\LL(2)\setminus\CC$, the $(\ref{mainx1})$ is {globally exactly controllable} in $H^{4+\epsilon}_{\Gi}$ with 
$\epsilon>0$.
\end{teorema}
\begin{proof} Let $r$ be the symmetry axis of $\Gi$  (see Figure \ref{parametrizzazione1}). We 
construct $(\phi_k)_{k\in\N^*}$ as a sequence of symmetric or skew-symmetric functions with respect to $r$.
If $\phi_k=(\phi_k^1,\phi_k^2)$ is skew-symmetric, then $\phi^2_k\equiv 0$, 
$\phi^1_k(0)=\phi^1_k(L_1/2)=\phi^1_k(L_1)= 0$ and $\dd_x\phi^1_k(0)=\dd_x\phi^1_k(L_1).$ We respectively denote 
by $$(f_k)_{k\in\N^*}=\left(\Big(\sqrt{\frac{2}{L_1}}\sin\Big(x\frac{2k\pi}{L_1}\Big),0\Big)\right)_{k\in\N^*},\ \ 
\ \ \  \ (\nu_k)_{k\in\N^*}:=\left(\frac{4 k^2\pi^2}{L_1^2}\right)_{k\in\N^*}$$
the skew-symmetric eigenfunctions belonging to the Hilbert base $(\phi_k)_{k\in\N^*}$ and the ordered sequence of 
corresponding eigenvalues. If $\phi_k=(\phi_k^1,\phi_k^2)$ is 
symmetric, then we have $\dd_x\phi^1_k(L_1/2)=0$ and $\phi_k^1(\cdot)=\phi_k^1(L_1-\cdot)$. The ($\Di$) conditions 
on $\widetilde v$ implies that the symmetric eigenfunctions corresponding to the eigenvalues 
$(\mu_k)_{k\in\N^*}$ are 
$$(g_k)_{k\in\N^*}:=\Big(\Big(a_k^1\cos\Big(\sqrt{\mu_k}\Big(x-\frac{L_1}{2}\Big)\Big),a_k^2\sin(\sqrt{\mu_k} 
x)\Big)\Big)_{k\in\N^*},\ \ \ \ \ \text{for}\ \ \ \ \  \{(a_k^1,a_k^2)\}_{k\in\N^*}\subset\C^2.$$
Now, we characterize the eigenvalues
$(\mu_k)_{k\in\N^*}$. The $(\NN\KK)$ conditions in $v$ ensure that
$a_k^1\cos(\sqrt{\mu_k}(L_1/2))=a_k^2\sin(\sqrt{\mu_k}L_2))$ and 
$2a_k^1\sin(\sqrt{\mu_k}(L_1/2))+a_k^2\cos(\sqrt{\mu_k}L_2))=0$. Finally, $(\mu_k)_{k\in\N^*}$ are the 
zeros of \begin{equation}\label{kirch}  
2\tan(\sqrt{\mu_k}(L_1/2))+\cot(\sqrt{\mu_k}L_2))=0.\end{equation}
The remaining part of the proof follows by the same argument of the one of Theorem \ref{bim.1}. The only 
difference 
is that we need to use Remark $\ref{zuazuafigo}$ instead of $\cite[Proposition\ A.11]{wave}$. \qedhere
\end{proof}

\begin{osss}\label{autovalori espliciti}
As showed in the proof of Theorem \ref{bim.1}, the study of the spectrum of $A$ on a $3$-edges star-graph equipped 
with ($\Di$) consists in seeking for $(\lambda_k)_{k\in\N}$ solving the first two identities of \eqref{tremor}. 
If for instance the lengths of the edges are equal to $L>0$, the eigenvalues are 
the zeros of $\sin(\sqrt{\lambda}L)$ and of $\cos(\sqrt{\lambda}L)$. However, in the general framework, the 
eigenvalues are obtained by solving $\sum_{l\leq 3}\cot(\sqrt{\lambda}L_l)=0$ which is a 
transcendental equation and then not always explicitly solvable. Similarly in Theorem \ref{bim.2}, some 
eigenvalues are 
the zeros of the transcendental equation \eqref{kirch}. The same observations is valid for other 
graphs (see 
\cite{wave} for further details).
\end{osss}

\section{Well-posedness and interpolation properties of the spaces $H^s_\Gi$}\label{well}

In the current section, we provide the {well-posedness} of the $(\ref{mainx1})$.

\begin{teorema}\label{laura}
Let $(A,B)$ satisfy Assumptions II$(\eta,\tilde d)$ with $\eta>0$ and $\tilde d\geq 0$. Let $\psi^0\in 
H^{2+d}_{\Gi}$ with $d$ introduced in Assumptions II. Let $u\in L^2((0,T),\R)$ with $T>0$. There
		exists a unique mild solution of ($\ref{mainx1}$) in
		$H^{2+d}_{\Gi}$, i.e. a function $\psi\in C_0([0,T],H^{2+d}_{\Gi})$ such that for every $t\in[0,T]$,
		\begin{equation}\label{form}
		\psi(t,x)=e^{-iA t}\psi^0(x)-
		i\int_0^t e^{-iA(t-s)}u(s)\psi(s,x)ds.
		\end{equation}
		Moreover, there exists $C=C(T,B,u)>0$ so that $\|\psi\|_{C^0([0,T],H^{2+d}_{\Gi})}\leq 
C\|\psi^0\|_{H^{2+d}_{\Gi}}$, while $\|\psi(t)\|_{L^2}=\|\psi^0\|_{L^2}$ for every $t\in[0,T]$ and $\psi_0\in 
H^{2+d}_{\Gi}.$
	\end{teorema}
Before proving Theorem 
$\ref{laura}$, we present some interpolation properties for the spaces $H^s_{\Gi}$ with $s\geq 0$ in the 
following proposition. In this result, we denote by $H^0_\Gi,\ H^0_{\NN\KK}$ and $H^0$ the Hilbert space 
$L^2(\Gi,\C)$.

%
%
\begin{prop}\label{bor}\text{ }
	
	\noindent
{\bf 1)} If the compact quantum graph $\Gi$ is equipped with ($\Di$/$\NN$), then 
$$H^{s_1+s_2}_{\Gi}=H_{\Gi}^{s_1}\cap H^{s_1+s_2} \ \ \ \text{for}\ \ \ s_1\in\N,\ s_2\in[0,1/2).$$
		
		\noindent
		{\bf 2)} If the compact quantum graph $\Gi$ is equipped with ($\NN$), then
		$$H^{s_1+s_2}_{\Gi}=H_{\Gi}^{s_1}\cap H^{s_1+s_2}_{\NN\KK} \ \ \ \text{for}\ \ \ s_1\in 2\N,\ 
s_2\in[0,3/2).$$
		
		\noindent
		{\bf 3)} If the compact quantum graph $\Gi$ is equipped with ($\Di$), then
		$$H^{s_1+s_2+1}_{\Gi}=H_{\Gi}^{s_1+1}\cap H^{s_1+s_2+1}_{\NN\KK} \ \ \ \text{for}\ \ \ s_1\in 2\N,\ 
s_2\in[0,3/2).$$
\end{prop}
\begin{proof} {\bf 1)  Graph equipped with ($\Di$/$\NN$).}
We start by proving the first statement of Proposition \ref{bor}. 

	\smallskip

	\needspace{3\baselineskip}

	\noindent
	{\bf 1) (a) Preliminaries.} Let $I^\NN$ and $I^\Di$ be two quantum graphs defined on an interval $I$ of 
length $L$. We suppose that $I^\NN$ is equipped with $(\NN)$, while $I^\Di$ is equipped with $(\Di)$. 
From $\cite[Definition \ 2.1]{boundaries}$, for every $s_1\in 2\N,$ $s_2\in[0,3/2)$ and $s_3\in [0,1/2)$, we have 
	\begin{equation}\begin{split}\label{chenon}H^{s_1+s_2}_{I^{\NN}}=H_{I^{\NN}}^{s_1}\cap H^{s_1+s_2} 
(I^{\NN},\C), \ \ \  H^{s_1+s_2+1}_{I^{\Di}}=H_{I^{\Di}}^{s_1+1}\cap H^{s_1+s_2+1} (I^{\Di},\C), \ \ \ 
H^{s_3}_{I^{\Di}}=H^{s_3} (I^{\Di},\C).\\
	\end{split}\end{equation}
Let $\Gi=I^{\MM}$ be an interval equipped with $(\NN)$ in the external vertex parametrized with $0$ and with 
($\Di$) in the other. We prove
	\begin{equation}\label{chenon2}
	H^{s_1+s_2}_{I^{\MM}}=H_{I^{\MM}}^{s_1}\cap H^{s_1+s_2} (I^{\MM},\C),\ \ \ \ \ \ \forall s_1\in \N,\ 
s_2\in[0,1/2).\\
	\end{equation}

\noindent
Let $\widetilde{I}^{\Di}$ and $\widetilde{I}^{\NN}$ respectively be two sub-intervals of $ {I}^{\MM}$ of length 
$\frac{3}{4}L$. The interval $\widetilde{I}^{\Di}$ contains one external vertex of $ {I}^{\MM}$, while 
$\widetilde{I}^{\NN}$ contains the other. We consider both the intervals as quantum graphs: $\widetilde{I}^{\Di}$  
is 
equipped in both the external vertices with $(\Di)$ and $\widetilde{I}^{\NN}$ is equipped with $(\NN)$.
Fixed $s>0$, $$H^s(I^\MM,\C)=H^s({\widetilde I^\Di},\C)\times H^s({\widetilde 
I^\NN},\C),\ \ \ \ \ \ \ \ \ L^2(I^\MM,\C)=L^2({\widetilde I^\Di},\C)\ \times\ L^2({\widetilde I^\NN},\C).$$
	Let $[\cdot , \cdot ]_{\theta}$ be the complex interpolation of spaces for $0<\theta<1$ defined in 
$\cite[\,Definition,\ Chapter\ 1.9.2]{schifo}$. From
	$\cite[\,Chapter\ 1.15.1\ \&\ Chapter\ 1.15.3]{schifo}$, for $s_2\in[0,1/2)$, we have $\big[L^2(\widetilde 
I^\NN,\C), H^2_{\widetilde I^\NN}\big]_{s_2/2}=H^{s_2}_{\widetilde I^\NN}$ and $\big[L^2(\widetilde 
I^\Di,\C), H^2_{\widetilde I^\Di}\big]_{s_2/2}=H^{s_2}_{\widetilde I^\Di}.$
	Thanks to $\cite[\,relation\ (12),\ Chapter\ 1.18.1]{schifo}$, we have $$\Big[L^2(\widetilde I^\NN,\C)\ 
\times\ L^2(\widetilde I^\Di,\C), H^2_{\widetilde I^\NN}\  \times\ H^2_{\widetilde 
I^\Di}\Big]_{{s_2}/2}=\Big[L^2(\widetilde I^\NN,\C),H^2_{\widetilde I^\NN}\Big]_{{s_2}/2}\ \times\ 
\Big[L^2(\widetilde I^\Di,\C), H^2_{\widetilde I^\Di}\Big]_{{s_2}/2},$$
	\begin{equation*}\begin{split}&\Longrightarrow\ \ \ \  H^{s_2}_{I^\MM}=\Big[L^2(I^\MM,\C), 
H^2_{I^\MM}\Big]_{{s_2}/2}=\Big[L^2(\widetilde I^\NN,\C),H^2_{\widetilde I^\NN}\Big]_{{s_2}/2}\ \times\ 
\Big[L^2(\widetilde I^\Di,\C), H^2_{\widetilde I^\Di}\Big]_{{s_2}/2}=H^{s_2}_{\widetilde I^\NN}\ \times\ 
H^{s_2}_{\widetilde I^\Di}.\end{split}\end{equation*}
Equivalently, $H^{s_1+s_2}_{I^\MM}=H^{s_1+s_2}_{\widetilde I^\NN}\times H^{s_1+s_2}_{\widetilde I^\Di}$ for every 
$s_1\in\N^*$ and $s_2\in[0,1/2)$ which leads to $(\ref{chenon2})$ thanks to $(\ref{chenon})$.

	\medskip

	\needspace{3\baselineskip}

	\noindent
	{\bf 1) (b) Sobolev's spaces for star graphs with equal edges.}  Let $I^\NN$ and $I^\MM$ be defined as in {\bf 
1) (a)}. We 
respectively call $A_\NN$ and $A_\MM$ the two self-adjoint Laplacians defining $I^\NN$ and $I^\MM$. Let 
$(f_j^{1})_{j\in\N^*}$ be a Hilbert basis of $L^2(I,\C)$ made by eigenfunctions of $A_\NN$ and 
$(f_j^{2})_{j\in\N^*}$ a Hilbert basis of $L^2(I,\C)$ composed by eigenfunctions of $A_\MM$. Let $\Gi$ be a 
star graph of $N$ edges long $L$ and equipped with ($\NN$). 
	The ($\NN$) conditions on $V_e$ imply that 
$$\phi_k=(a^1_k\cos(x\sqrt{\lambda}_k),...,a_j^N\cos(x\sqrt{\lambda}_k)),\ \ \ \ \forall k\in\N^*$$ where 
$\lambda_k$ is the corresponding eigenvalue and for suitable $\{a_k^l\}_{l\leq N}\subset\C$. 
The ($\NN\KK$) condition in $V_i$ ensures that $$\sin(\sqrt{\lambda}_kL)\sum_{l\leq 
N}a^l_k=0,\ \ \ \ \ \ \ \ a^1_k\cos(\sqrt{\lambda}_kL)=...=a^N_k\cos(\sqrt{\lambda}_kL),\ \ \ \ \ \forall 
k\in\N^*.$$

\vspace{-2mm}

\noindent
Thus, each eigenvalue is either of the form $\frac{(n-1)^2\pi^2}{L^2}$ when $\sum_{l\leq N}a^l_k\neq 
0$, or $\frac{(2n-1)^2\pi^2}{4L^2}$ when $\sum_{l\leq N}a^l_k=0$ for suitable $n\in\N^*$. 
	For every $k\in\N^*$, there exists $j(k)\in\N^*$ such that 
	\begin{equation}\begin{split}\label{iu} \phi_k^l\ \ \ \text {   is equal  either to  }\ \ \ c_k^l 
f^1_{j(k)},\ \ \ \text{   or  to }\ \ \ c_k^l f^2_{j(k)}\ \ \ \text{      with      }\ \ \ c_k^l\in\C,\  
\ |c_k^l|\leq 1,\ \ \ \forall l\in\{1,...,N\}.\\
	\end{split}\end{equation}
	In addition, for each $k\in\N^*$ and $m\in\{1,2\}$, there exist $\widetilde j(k)\in\N^*$ and $l\leq N$ such 
that $f^m_{k}=c_{\widetilde j(k)}^{l}\phi_{\widetilde j(k)}^{l}$ with $c_{\widetilde j(k)}^{l}\in\C$ uniformly 
bounded in $k\in\N^*$ and $l\leq N$. The last identity and $(\ref{iu})$ tell that the components of the 
elements $(\phi_k)_{k\in\N^*}$ are elements of $(f_j^{1})_{j\in\N^*}$ and $(f_j^{2})_{j\in\N^*}$ and vice versa. 
Thus, 
$\psi=(\psi^1,...,\psi^N)\in H^s_\Gi$ if and only if $\psi^l\in 
H^s_{I^\NN}\cap H^s_{I^\MM}$ for every $l\leq N.$ 

\medskip
	
	\needspace{3\baselineskip}

	\noindent
	{\bf 1) (c) Conclusion.} Let $\Gi$ be equipped with ($\Di$/$\NN$) and $\widetilde 
L<\min\{L_k/2:k\in\{1,...,N\}\}$.
		Let $n(v)$ be defined in $(\ref{molteplicit})$ for every $v\in V_e\cup V_i$. We define the graphs 
$\widetilde \Gi(v)$ for every $v\in V_i\cup V_e$ and the intervals $\{I_j\}_{j\leq N}$ as follows (see Figure 
$\ref{spezzino}$ for an explicit example). If $v\in V_i$, then $\widetilde \Gi(v)$ is a star sub-graph of $\Gi$ 
equipped with ($\NN$) and composed by $n(v)$ edges long $\widetilde L$ and connected to the internal vertex $v$. 
If 
$v\in V_e$, then $\widetilde \Gi(v)$ is an interval long $\widetilde L$ such that the external vertex $v$ is 
equipped with the same boundary conditions that $v$ has in $\Gi$. We impose $(\NN)$ on the other vertex. 
Let $v,\hat v\in V_e\cup V_i$ be such that $v, \hat v \in {e_1}.$ Now, the graphs 
$\widetilde\Gi(v)$ and $\widetilde\Gi(\hat v)$ have respectively two 
external vertices $w_1$ and $w_2$ lying on the same edge $e_1$ and such that $w_1\not\in\widetilde\Gi(\hat v)$. We 
construct an interval $I_1$ strictly containing $w_1$ and $w_2$, strictly contained in $e_1$ and equipped with 
($\NN$). We repeat the procedure for every edge $e_j$ with $j\leq N$ and we define the intervals 
$\{I_j\}_{j\leq N}$.  
	\begin{figure}[H]
	\centering
	\includegraphics[width=\textwidth-150pt]{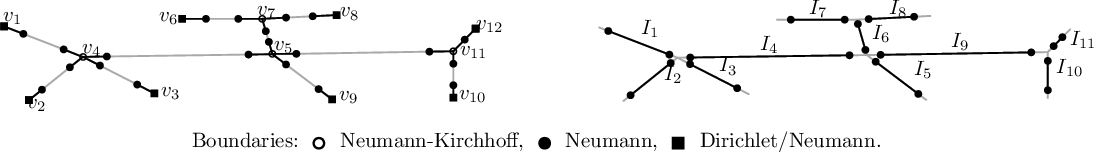}
	\caption{The left and the right figures respectively represent the graphs $\{\widetilde \Gi(v)\}_{v\in V_i\cup 
V_e}$ and the intervals $\{I_j\}_{j\leq N}$ for a given graph $\Gi$. }\label{spezzino}
\end{figure}

\noindent
From {\bf 1) (a)} and {\bf 1)  (b)}, for every $v\in V_i\cup V_e$, $j\leq N$, $s_1\in \N$ and $s_2\in[0,1/2)$, we 
have 
the validity of the identities $H^{s_1+s_2}_{\widetilde \Gi(v)}=H_{\widetilde \Gi(v)}^{s_1}\cap H^{s_1+s_2} 
(\widetilde \Gi(v),\C)$ and $H^{s_1+s_2}_{I_j}=H_{I_j}^{s_1}\cap H^{s_1+s_2} (I_j,\C).$
	We notice that $G:=\{\widetilde\Gi(v_j)\}_{ j\leq M}\cup\{I_j\}_{j\leq N}$ covers $\Gi$. As in {\bf 1) (a)}, 
we 
see each function of domain $\Gi$ as a vector of functions of domain $G_j$ with $j\leq M+N$. The first relation of 
Proposition \ref{bor} is proved by adopting $\cite[relation\ (12),\ Chapter\ 1.18.1]{schifo}$ as in  
{\bf 1) (a)}.

	\smallskip
	\needspace{3\baselineskip}
	
	\noindent
	
	\noindent
	{\bf 2) Graphs equipped with ($\NN$).}  Let $\Gi$ be equipped with ($\NN$) and $N_e=|V_e|$. 
We 
consider $\{\widetilde \Gi(v)\}_{v\in V_e}$ introduced in {\bf 1) (c)} and we define $\widetilde \Gi$ from $\Gi$ 
as 
follows (see Figure $\ref{spezzino1}$). 
	For every $v\in V_e$, we remove from the edge including $v$, a section of length $\widetilde L/2$ containing 
$v$. We equip the new external vertex with ($\NN$).
	\begin{figure}[H]
		\centering
		\includegraphics[width=\textwidth-150pt]{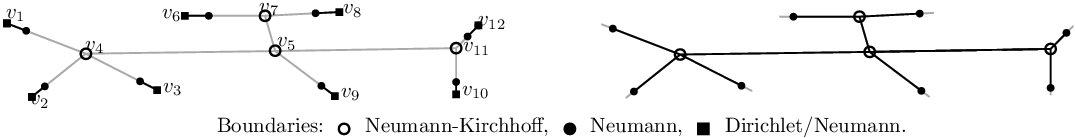}
		\caption{The left and the right figures respectively represent the graphs $\{\widetilde \Gi(v)\}_{v\in 
V_e}$ and $\widetilde \Gi$ for a given graph $\Gi$.}\label{spezzino1}
	\end{figure}

	\noindent
	We call  $G':=\{G'_j\}_{j\leq N_e+1}:=\{\widetilde\Gi(v)\}_{v\in V_e}\cup\{\widetilde\Gi\}$ which covers 
$\Gi$. 
For every $s_1\in 2\N,\ s_2\in[0,3/2)$, we have $H^{s_1+s_2}_{\widetilde\Gi(v)}=H_{\widetilde\Gi(v)}^{s_1}\cap 
H^{s_1+s_2}$ from $(\ref{chenon})$. Now,
$H^{s_1+s_2}_{\NN\KK}=H^{s_1+s_2}_{\widetilde\Gi}\ \times \prod_{v\in V_e}H^{s_1+s_2}(\widetilde\Gi(v),\C)$. The 
second relation of Proposition \ref{bor} follows from the arguments of {\bf 1) (a)}.

	\smallskip
	\needspace{3\baselineskip}
	
	\noindent
	
	\noindent
	{\bf 3) Graphs equipped with ($\Di$).}
The third point of Proposition 
\ref{bor} is proved as the second by considering $\{\widetilde\Gi(v)\}_{v\in V_e}$ as intervals equipped with 
$(\Di)$ and $\widetilde\Gi$ equipped with $(\Di)$ in its external vertices.\qedhere
\end{proof}

We are finally ready to prove Theorem \ref{laura} which follows from Proposition \ref{bor} and from the 
following 
auxiliary result.

\begin{lemma}\label{uppp}
	Let $(\lambda_k)_{k\in\N*}$ be the eigenvalues of a self-adjoint Laplacian $A$ defined on a compact quantum 
graph 
equipped with ($\Di$), ($\NN$) or ($\Di/\NN$). For every $T>0$, there 
exists $C(T)>0$ uniformly bounded for $T$ lying on bounded intervals such that
	$$\forall g\in L^2((0,T),\C),\ \ \ \ \ \  \ \left\|\int_0^T 
	e^{i\lambda_{(\cdot)} s}g(s) ds\right\|_{\ell^2}\leq C(T)\|g\|_{L^2((0,T),\C)}.$$
\end{lemma}
\begin{proof}
The result is a consequence of Proposition \ref{uppp123} which validity is ensured by Lemma
\ref{g13}.\qedhere
\end{proof}

\begin{proof}[Proof of Theorem $\ref{laura}$]	\noindent
	{\bf 1) Preliminaries. }  Let $T >0$ and the function $f$ be such that $f(s)\in H^{2+d}\cap 
H^{1+d}_{\NN\KK}\cap H^2_{\Gi}$ for almost every $s\in (0,T)$. We introduce 
$$G(\cdot):=\int_0^{(\cdot)}e^{iA\tau} f(\tau) d\tau. $$ In the first part of the proof, we prove that $G \in 
C^0([0,T], H^{2+d}_{\Gi})$ by ensuring the existence of $C(T)>0$ uniformly bounded for $T$ lying on bounded 
intervals such that $\|G\|_{L^{\infty}((0,T),H^{2+d}_{\Gi})}\leq C(T)\| f\|_{L^2((0,T),H^{2+d})}.$
To the purpose, we distinguish the different frameworks described by Assumptions II.

	\smallskip
	
	\needspace{3\baselineskip}

	\noindent
	{\bf 1) (a) Under Assumptions II.1 .} Let $f(s)\in H^{3}\cap H_{\Gi}^2$ for almost every $s\in (0,T)$ 
and $f(s)=(f^1(s),...,f^N(s))$. We prove that $G\in C^0([0,T], H^3_{\Gi})$. Let $t\in(0,T)$. The definition of 
$G(t)$ implies
	\begin{equation}\label{formuletta}\begin{split}
	G(t)=\sum_{k=1}^\infty\phi_k \int_0^t e^{i\lambda_ks}\la\phi_k,f(s)\ra_{L^2} ds,\ \ \ \ \ \  \ \ \ 
\|G(t)\|_{(3)}=\Big(\sum_{k\in\N^*}\Big|k^3\int_0^t e^{i\lambda_ks}\la\phi_k,f(s)\ra_{L^2} 
ds\Big|^2\Big)^\frac{1}{2}.\\
	\end{split}\end{equation}
	We estimate $\la\phi_k,f(s,\cdot)\ra_{L^2}$ for each $k\in\N^*$ and $s\in(0,t)$. We suppose that $\lambda_1\neq 
0$. Let $\dd_x f(s)=(\dd_x f^1(s),...,\dd_x f^N(s)$ be the derivative of $f(s)$. We call $\dd e$ the two points 
composing the boundaries of an edge 
$e$. For every $v\in V_e$, $\tilde v	\in V_i$ and $j\in N(\tilde v)$, there exist $a(v),a^j(\tilde 
v)\in\{-1,+1\}$ such that 
	\begin{equation}\label{compu}\begin{split}
	&\la\phi_k,f(s)\ra_{L^2}
	=
	\frac{1}{\lambda_k^2 }\int_\Gi\dd_{x}\phi_k(y)\dd_{x}^3f(s,y)dy+\frac{1}{\lambda_k^2}\sum_{v\in V_i\cup V_e} 
	\sum_{j\in N(v)}a^j(v)\dd_{x}\phi_k^j(v)\dd_{x}^2f^j(s,v).\\
	\end{split}\end{equation}
	From Lemma \ref{interessante}, there exists $C_1>0$ such that $\lambda_k^{-2}\leq C_1 k^{-4}$ for 
every 
$k\in\N^*$ and
	\begin{equation}\label{comput}\begin{split}
	\left|k^3\int_0^t e^{i\lambda_ks}\la\phi_k,f(s)\ra_{L^2} ds\right|
	\leq\frac{C_1}{k}\Bigg( &\sum_{v\in V_i\cup V_e} 
	\sum_{j\in N(v)}\left|\dd_{x}\phi_k^j(v)\int_0^t e^{i\lambda_ks}\dd_{x}^2f^j(s,v)ds\right|\\
	&+
	\left|\int_0^t e^{i\lambda_ks}\int_\Gi\dd_{x}\phi_k(y)\dd_{x}^3f(s,y)dyds\right|\Bigg).\\
	\end{split}\end{equation}
	\begin{osss}\label{qqq}
		We point out that $A'\lambda_k^{-1/2}\dd_{x}\phi_k=\lambda_k\lambda_k^{-1/2}\dd_{x}\phi_k$ for every 
$k\in\N^*,$ where $A'=-\Delta$ is a self-adjoint Laplacian with compact resolvent. 
		Thus, 
$\|\lambda_k^{-1/2}\dd_{x}\phi_k\|_{L^2}^2=\la\lambda_k^{-1/2}\dd_{x}\phi_k,\lambda_k^{-1/2}\dd_{x}\phi_k\ra_{L^2}
=\la \phi_k,\lambda_k^{-1}A\phi_k\ra_{L^2}=1$ and then $\big(\lambda_k^{-1/2}\dd_{x}\phi_k\big)_{k\in\N^*}$ is a 
Hilbert basis of $L^2(\Gi,\C)$.
	\end{osss}

	\noindent
	Let ${\bf a^l}=(a_k^l)_{k\in\N^*},{\bf b^l}=(b_k^l)_{k\in\N^*}\subset\C$ for $l\leq N$ be so that 
$\phi_k^l(x)=a_k^l\cos(\sqrt{\lambda}_kx)+ b_k^l\sin(\sqrt{\lambda}_kx)$ and $-a_k^l\sin(\sqrt{\lambda}_kx)+ 
b_k^l\cos(\sqrt{\lambda}_kx)=\lambda_k^{-1/2}\dd_{x}\phi_k^l(x).$ There holds ${\bf a^l},{\bf 
b^l}\in\ell^\infty(\C)$ since $$2\geq 
\|\lambda_k^{-1/2}\dd_{x}\phi_k^l\|_{L^2(e^l)}^2+\|\phi_k^l\|_{L^2(e^l)}^2=(|a_k^l|^2+|b_k^l|^2)|e_l|,\ \ \  \ \ \ 
\ \forall k\in\N^*,\ l\leq  N.$$ Thus, there exists $C_2>0$ so that, for every $k\in\N^*$ and $v\in V_e\cup V_i$, 
we have $|\lambda_k^{-1/2}\dd_{x}\phi_k(v)|\leq C_2$. From the validity of the relations $(\ref{formuletta})$ and 
$(\ref{comput})$, it follows
	\begin{equation*}\begin{split}
	\|G(t)\|_{(3)}&\leq C_1C_2\sum_{v\in V_e\cup V_i}\sum_{j\in N(v)}\Big\|\int_0^t\dd_{x}^2 f^j(s,v)e^{i \lambda_{(\cdot)}s}ds \Big\|_{\ell^2}+C_1\Big\|\int_0^t\big\la {\lambda_{(\cdot)}^{-1/2}}\dd_x\phi_{(\cdot)}(s),\dd_{x}^3 f(s)\big\ra_{L^2} e^{i \lambda_{(\cdot)}s}ds\Big\|_{\ell^2}.
	\end{split}\end{equation*}
	The last relation, Lemma $\ref{uppp}$ and Remark \ref{qqq} ensure the existence of $C_3(t),C_4(t)>0$ 
uniformly bounded for $t$ 
in bounded intervals such that
	
	\begin{equation}\label{GG2}\begin{split}
	\|G\|_{H^3_{\Gi}} 
	&\leq C_3(t)\sum_{v\in V_e\cup V_i}\sum_{j\in N(v)}\|\dd_{x}^2 f^j(\cdot,v)\|_{L^2((0,t),\C)}+\sqrt{t}\|f\|_{L^2((0,t),H^3)}\leq C_4(t)\|f(\cdot,\cdot)\|_{L^2((0,t),H^3)}.
	\end{split}\end{equation}
We underline that the identity is also valid when $\lambda_1=0$, which is proved by isolating the term with $k=1$ and by repeating the steps above.
For every $t\in [0,T]$, the inequality $(\ref{GG2})$ shows that $G(t)\in H^3_\Gi$. The provided upper bounds are 
uniform and the Dominated Convergence Theorem leads to $G\in C^0([0,T], H^3_{\Gi}).$
When $f(s)\in H^{5}\cap H_{\Gi}^4$ for almost every $s\in (0,T)$, the techniques just adopted 
leads to $G\in C^0([0,T], H^5_{\Gi}).$ 
	\smallskip

	Let $F(f)(t):=\int_0^{t}e^{iA\tau} f(\tau) d\tau$. The first part of the proof implies
	$$F: L^2((0,T),H^{3}\cap H_{\Gi}^2)\longrightarrow C^0([0,T], H^3_{\Gi}),\ \ \ \ \ \ \ F: L^2((0,T),H^{5}\cap H_{\Gi}^4)\longrightarrow C^0([0,T], H^5_{\Gi}).$$
	From a classical interpolation result (see $\cite[Theorem\ 4.4.1]{bergolo}$ with $n=1$), we have $F: 
L^2((0,T),H^{2+d}\cap H_{\Gi}^{1+d})\longrightarrow C^0([0,T], H^{2+d}_{\Gi})$ with $d\in [1,3]$. Thanks to Proposition 
$\ref{bor}$, if $d\in [1,3/2)$ and $f(s)\in H^{2+d}\cap H^{1+d}_{\NN\KK}\cap H^2_{\Gi}=H^{2+d}\cap H^{1+d}_{\Gi}$ 
for almost every $s\in (0,T)$, then $G\in C^0([0,T], H^{2+d}_{\Gi}).$ 
	
	\medskip
	\needspace{3\baselineskip}

	\noindent
	{\bf 1) (b) Under Assumptions II.3 .} If $\Gi$ is equipped with ($\Di$), then $H_{\Gi}^2= H_{\NN\KK}^2\cap 
H_{\Gi}^1$
	and $H_{\Gi}^4=H_{\NN\KK}^4\cap H_{\Gi}^3$ from Proposition $\ref{bor}$. As above, if $f(s)\in H^{3}\cap 
H_{\NN\KK}^2\cap H_{\Gi}^1$ for almost every $s\in (0,T)$, then $G\in C^0([0,T], H^3_{\Gi})$, while 
if $f(s)\in H^{5}\cap H_{\NN\KK}^4\cap H_{\Gi}^3$ for almost every $s\in (0,T)$, then $G\in 
C^0([0,T], H^5_{\Gi})$. 
From the interpolation techniques, if $d\in [1,5/2)$ and $f(s)\in H^{2+d}\cap H^{1+d}_{\NN\KK}\cap H^{d}_{\Gi}$ for 
almost every $s\in (0,T)$, then $G\in C^0([0,T], H^{2+d}_{\Gi})$.
	
	\medskip 
	
	\needspace{3\baselineskip}

	\noindent
	{\bf 1) (c) Under Assumptions II.2 .} Let $f(s)\in H^{4}\cap H_{\NN\KK}^3\cap H_\Gi^2$ for almost every $s\in 
(0,T)$ and $\Gi$ be equipped with $(\NN)$. In this framework, the last term in right-hand side $(\ref{compu})$ is 
zero. Indeed, $\dd_x^2f(s)\in C^0$ as $f (s)\in H_{\NN\KK}^3$ and, for $v\in V_e$, we have $\dd_{x}\phi_k(v)=0$ 
thanks to the ($\NN$) boundary conditions (the terms $a^j(v)$ have different signs according to the orientation of 
the edges connected in $v$). For every $v\in V_i$, thanks to the $(\NN\KK)$ in $v\in V_i$, we have $\sum_{j\in 
N(v)}a^j(v)\dd_{x}\phi_k^j(v)=0$. 
	From $(\ref{compu})$, we obtain
	\begin{equation*}\begin{split}
	\la\phi_k,f(s)\ra_{L^2}&
	=
	-\frac{1}{\lambda_k^2}\sum_{v\in V_i\cup V_e} 
	\sum_{j\in N(v)}a^j(v)\phi_k^j(v)\dd_{x}^3f^j(s,v)+
	\frac{1}{\lambda_k^2 }\int_\Gi\phi_k(y)\dd_{x}^4f(s,y)dy.\\
	\end{split}\end{equation*}
	Now, $(\phi_k)_{k\in\N^*}$ is a Hilbert basis of $L^2(\Gi,\C)$ and we proceed as in $(\ref{comput})$ and 
$(\ref{GG2})$. From Lemma $\ref{uppp}$, there exists $C_6(t)>0$ uniformly bounded for $t$ lying in bounded 
intervals such that $\|G\|_{H^4_{\Gi}}\leq C_1(t)\|f(\cdot,\cdot)\|_{L^2((0,t),H^4)}$ and $G\in C^0([0,T], 
H^{4}_{\Gi}).$ Equivalently, when $f(s)\in H^{6}\cap H_{\NN\KK}^5\cap H_\Gi^4$ for almost every $s\in (0,T)$, we 
have $G\in C^0([0,T], H^{6}_{\Gi}).$
	As above, Proposition $\ref{bor}$ implies that when $d\in[2,7/2)$ and $f(s)\in H^{2+d}\cap H^{1+d}_{\NN\KK}\cap 
H^2_{\Gi}$ for almost every $s\in (0,T)$, then $G\in C^0([0,T], H^{2+d}_{\Gi}).$

	\smallskip
	
	\needspace{3\baselineskip}
	\noindent
	{\bf 2) Conclusion. } As $Ran(B|_{H_{\Gi}^{2+d}})\subseteq H^{2+d}\cap H^{1+d}_{\NN\KK} H^2_{\Gi}\subseteq H^{2+d}$, we have $B\in L(H^{2+d}_{\Gi},H^{2+d})$ thanks to the arguments of $\cite[Remark\ 2.1]{mio1}$. Let $\psi_0\in H_{\Gi}^{2+d}$. We consider the map $F:\psi \in C^0([0,T],H^{2+d}_{\Gi})\mapsto\phi\in C^0([0,T],H^{2+d}_{\Gi}) $ with
	$$\phi(t)=F(\psi)(t)=e^{-iA t}\psi_0-\int_0^te^{-iA(t-s)}u(s) B\psi(s)ds,\ \ \ \ \ \forall t\in 
[0,T].$$
	For every $\psi^1,\psi^2\in  C^0([0,T], H^{2+d}_{\Gi})$, we have $F(\psi^1)(t)-F(\psi^2)(t)=\int_0^t e^{-iA(t-s)}u(s) B(\psi^1(s)-\psi^2(s))ds$. From {\bf 1)}, there exists $C(t)>0$ uniformly bounded for $t$ lying on bounded intervals such that
	\begin{equation*}\begin{split}
	&\|F(\psi^1)-F(\psi^2)\|_{L^\infty((0,T),H^{2+d}_\Gi)}\leq C(T)\|u\|_{L^2((0,T),\R)}\iii B\iii_{L(H^{2+d}_{\Gi},H^{2+d})} \|\psi^1-\psi^2\|_{L^\infty((0,T),H^{2+d}_\Gi)}.\\
	\end{split}\end{equation*}
	 If $\|u\|_{L^2((0,T),\R)}$ is small enough, then $F$ is a contraction and Banach Fixed Point Theorem implies 
that there exists $\psi \in C^0([0,T],H^{2+d}_{\Gi}) $ such that $F(\psi)=\psi.$ When $\|u\|_{L^2((0,T),\R)}$ is 
not sufficiently small, one considers $\{t_j\}_{0\leq j\leq n}$ a partition of $[0,T]$ with $n\in\N^*$. We choose a 
partition such that each $\|u\|_{L^2([t_{j-1},t_j],\R)}$ is so small that the map $F$, defined on the interval 
$[t_{j-1},t_j]$, is a contraction. Thanks to the Banach Fixed Point Theorem, the existence and the uniqueness of 
the mild solution is provided. In conclusion, the solution $\psi$ of the \eqref{mainx1} when $u\in 
C^0((0,T),\R)$ is $C^1((0,T),L^2(\Gi,\C))$ 
and $\dd_t\|\psi(t)\|^2=0$, which implies $\|\psi(t)\| =\|\psi(0)\|$ for every $t\in[0,T]$. The generalization for 
$u\in L^2((0,T),\R)$ follows from classical density arguments. \qedhere
\end{proof}

\section{Abstract global exact controllability result}\label{proofglobal}
\subsection{Local controllability}
The aim of this section is to prove Theorem $\ref{global}$. The result is achieved by gathering the 
local exact controllability and the global approximate controllability (both provided below) thanks to the time 
reversibility of the $(\ref{mainx1})$. Before stating the local result, we need to introduce the following 
auxiliary lemma.

\begin{lemma}\label{momentproblem}
	Let the hypotheses of Theorem \ref{global} be satisfied. Let $T>2\pi/\delta$ with $\delta>0$ defined in Lemma 
\ref{g13}. For every $(x_k)_{k\in\N^*}\in h^{\tilde d}(\C)$ with $x_1\in\R$, there exists $u\in 
L^2((0,T),\R)$ such that 
	\begin{equation*}
	{x_{k}}=\int_{0}^Tu(\tau)e^{i(\lambda_k-\lambda_1) \tau}d\tau\ \ \ \ \ \ \ \ 
\forall k\in\N^*.\end{equation*}
\end{lemma}
\begin{proof}
The result is consequence of Proposition \ref{momentproblem123}.\qedhere
\end{proof}

\begin{prop}\label{lcl}
	Let the hypotheses of Theorem \ref{global} be satisfied. Let $s=2+d$ with $d$ defined in Assumptions II. There 
exist $T>0$ and $\epsilon>0$ such that, for every $$\psi\in O_{\epsilon,T}^{s}:=\big\{\psi\in H_{\Gi}^{s}\ \big|\ 
\|\psi\|_{L^2}=1,\ \|\psi -\phi_1(T)\|_{(s)}<\epsilon\big\},$$ there exists a control function $u\in 
L^2((0,T),\R)$ such that $\psi= \G^u_T\phi_1.$
\end{prop} 
\begin{proof}
The result can be proved by ensuring to the surjectivity, for $T>0$ sufficiently large, of the map
$$\G_T^{(\cdot)}\phi_1:u\in L^2((0,T),\R)\longmapsto \psi \in O_{\epsilon,T}^{s}\subset  H^s_{\Gi},\ \ \ 
\ \ \  \G_{t}^{(\cdot)}\phi_1=\sum_{k\in\N^*}{\phi_k(t)}\la \phi_k(t),\G_{t}^{(\cdot)}\phi_1\ra_{L^2}.$$ 

\vspace{-2mm}

\noindent
Let the map $\alpha$ be the sequence with elements $\alpha_{k}(u)=\la \phi_k(T), \G_{T}^u\phi_1\ra_{L^2}$ for 
$k\in\N^*$,
so that
$$\alpha:L^2((0,T),\R)\longrightarrow Q:=\{{\bf x}:=(x_k)_{k\in\N^*}\in h^s(\C)\ |\ \|{\bf x}\|_{\ell^2}=1\}.$$ 
The local controllability can be guaranteed by proving the local surjectivity of the map $\alpha$ in a 
neighborhood of $\alpha(0)=\updelta=(\delta_{k,1})_{k\in\N^*}$ with respect to the $h^s$ norm. To this end, we use 
the Generalized Inverse Function Theorem (\cite[Theorem\  1;\ p.\ 240]{Inv}) and we study the surjectivity of 
$\gamma(v):=(d_u\alpha(0))\cdot\ v$ the Fr�chet derivative of $\alpha$. Let $B_{j,k}:=\la\phi_j,B\phi_k\ra_{L^2}$ 
with $j,k\in\N^*$. The map $\gamma:L^2((0,T),\R)\longrightarrow T_{\updelta}Q=\{{\bf x}:=(x_k)_{k\in\N^*}\in 
h^s(\C)\ |\ ix_1\in\R\}$ is the sequence of elements $\gamma_{k}(v):=
-i\int_{0}^Tv(\tau)e^{i(\lambda_k-\lambda_1)s}d\tau B_{k,1}$ with $k\in\N^*$.
Now,
\begin{equation}\begin{split}\label{mome1}
{x_{k}}/{B_{k,1}}=-i\int_{0}^Tu(\tau)e^{i(\lambda_k-\lambda_1)\tau}d\tau,\ \ \ \ \ \ \ \ \ \ \forall 
(x_{k})_{k\in\N^*}\in T_{\updelta}Q\subset h^s 
\\
\end{split}\end{equation} is the moment problem associated to the local exact controllability. Proving 
surjectivity of $\gamma$ corresponds to ensure the solvability of \eqref{mome1}. In other words, we prove that 
there exists $T>0$ large enough such that, for every $(x_k)_{k\in\N^*}\in T_{\updelta}Q$, there exists $u\in 
L^2((0,T),\R)$ such that $(x_k)_{k\in\N^*}=(\gamma_{k} (u))_{k\in\N^*}$.
Even though the strategy of the proof is common for this kind of works (see 
$\cite{laurent,morgane1,morganerse2,mio1,mio2}$), proving the solvability of \eqref{mome1} can not be approached 
with the classical techniques as we can not ensure the validity of the spectral gap 
$\inf_{k\in\N^*}|\lambda_{k+1}-\lambda_k|>0$. To this purpose, we refer 
to the theory developed in Appendix \ref{momenti} which leads to Lemma \ref{momentproblem}.
We notice that $B_{1,1}\in\R$ as $B$ is symmetric, $i{x_{1}}/B_{1,1}\in\R$ and $\big(x_k 
/B_{k,l}\big)_{k\in\N^*}\in h^{d-\eta}\subseteq h^{\tilde d}$ thanks to the first point of Assumptions I. Thanks to 
Lemma \ref{g13} and the identity \eqref{gapp}, the hypotheses of Lemma $\ref{momentproblem}$ are satisfied and the 
solvability of $(\ref{mome1})$ is guaranteed in $h^{\tilde d}$. In conclusion, the map $\gamma$ is surjective and 
$\alpha$ is locally surjective, which implies the local exact controllability.\qedhere
\end{proof}

\begin{osss}\label{time}
The identity
\eqref{g13_aus} ensures the 
validity of \eqref{numer} for $\delta>0$ as large as desired when $\MM\in\N^*$ is also 
sufficiently large. As a consequence, Lemma \ref{momentproblem} is valid for any $T>0$ and the same is true for 
the solvability of the moment problem \eqref{mome1}. Finally, Proposition 
\ref{lcl} can be guaranteed for any positive time.
\end{osss}

\subsection{Global approximate controllability in $H^s_{\Gi}$}

\begin{defi}
	The \eqref{mainx1} is said to be globally approximately controllable in $H_{\Gi}^{s}$ with $s>0$ when, for 
every $\psi\in H^{s}_{\Gi}$, $\widehat\G\in U(L^2(\Gi,\C))$ \big(the space of the unitary operators in 
$L^2(\Gi,\C)$\big) such that $\widehat\G\psi\in H^{s}_{\Gi}$ and 
$\epsilon>0$, there exist $T>0$ and $u\in L^2((0,T),\R)$ such that $\|\widehat\G\psi-\G^u_T\psi\|_{(s)}<\epsilon$.
\end{defi}
\begin{prop}\label{approx}
Let $(A,B)$ satisfy Assumptions I$(\eta)$ and Assumptions II$(\eta,\tilde d)$ for $\eta>0$ and $\widetilde d\geq 0$. The (\ref{mainx1}) is globally approximately controllable in $H^{s}_{\Gi}$ for $s=2+d$ with $d$ from Assumptions II.
\end{prop}
\begin{proof}
The proof is obtained by simply adapting the one of \cite[Theorem\ 4.4]{mio1}. As a consequence, we only focus on 
 detailing those steps where the two proofs differ.

\smallskip

\needspace{3\baselineskip}

\noindent
{\bf 1) }As in the mentioned proof, in the point {\bf 1)\,\,}of the proof, we 
suppose that $(A,B)$ admits a non-degenerate chain of connectedness (see \cite[Section\ 4.2]{chambrion} or 
\cite[Definition\ 3]{nabile}). We treat 
the general case in the point {\bf 2)\,\,}. Let $\pi_m$ be the orthogonal projector 
$\pi_m:\Hi:=L^2(\Gi,\C)\longrightarrow \Hi_m:=\spn\{\phi_j\ :\ j\leq m\}$ with $m\in\N^*.$
Up to reordering $(\phi_k)_{k\in\N^*}$, the couples $(\pi_{m}A\pi_{m},\pi_{m}B\pi_{m})$ for $m\in\N^*$ admit 
non-degenerate chains of connectedness in $\Hi_{m}$. Let $\|\cdot\|_{BV(T)}=\|\cdot\|_{BV((0,T),\R)}$ and 
$\iii\cdot\iii_{(s)}:=\iii\cdot\iii_{L(H^s_{\Gi},H^s_{\Gi})}$ for $s>0.$

\smallskip

\needspace{3\baselineskip}

\noindent
{\bf 1) (a) Approximate controllability with respect to the $L^2$-norm.} Let $\psi\in \Hi$ and $\widehat \G\in 
U(\Hi)$. We refer to the proof of the global approximate controllability with 
respect to the $L^2$-norm developed in the first point of the proof of \cite[Theorem\ 4.4]{mio1}. By considering 
$\Hi:=L^2(\Gi,\C)$, the mentioned proof ensures the existence of $K_1,K_2,K_3>0$ such that for every 
$\varepsilon>0$, there exist $T>0$ 
and 
$u\in L^2((0,T),\R)$ such that  
\begin{equation}\begin{split}\label{casinooo}&\|u\|_{BV(T)} \leq K_1,\ \ \|u\|_{L^\infty((0,T),\R)}\leq 
K_2,\ \ T\|u\|_{L^\infty((0,T),\R)}\leq K_3\ \ \ \  \text{ and }\ \ \  \ 
\|\G_{T}^{u}\psi-\widehat\G\psi\|_{L^2}<\varepsilon.
\end{split}\end{equation}

\needspace{3\baselineskip}

\noindent
	{\bf 1) (b) Global approximate controllability in higher regularity norm.} Let $\psi\in H^{s}_{\Gi}$ with 
$s\in[s_1,s_1+2)$ and $s_1\in\N^*$. Let $\widehat \G\in U(\Hi)$ be such that $\widehat \G\psi\in H^{s}_{\Gi}$ and 
$B:H^{s_1}_{\Gi}\longrightarrow H^{s_1}_{\Gi}$. As in the proof of \cite[Theorem\ 4.4]{mio1}, we consider the 
propagation of regularity developed by Kato in 
$\cite{kato1}$ which ensures the following fact.
For every $T>0$, $u\in BV((0,T),\R)$ and $\psi\in H^{s_1+2}_{\Gi}$, there exists $C(K)>0$ depending on
\begin{equation}\label{diid1}K=\big(\|u\|_{BV(T)},\|u\|_{L^\infty((0,T),\R)},T\|u\|_{L^\infty((0,T),\R)}\big)\ 
\ \ \ \text{such that}\ \ \ \ \|\G_ { T}^{u} \psi\|_{(s_1+2)}\leq  C(K)\|\psi\|_{(s_1+2)}.\end{equation}
Now, we notice that, for every $\psi\in H^{6}_{\Gi}$, from the Cauchy-Schwarz inequality, we have $\|A\psi\|_{L^2}^2
\leq\|\psi\|_{L^2}\|A^2\psi\|_{L^2}$ and there exists $C_2>0$ such that $\|A^2\psi\|_{L^2}^4\leq\|A\psi\|_{L^2}^2\|A^3\psi\|_{L^2}^2\leq C_2\|\psi\|_{L^2}\|A^3\psi\|_{L^2}^3$. By following the same idea, for every $\psi\in H^{s_1+2}_{\Gi}$, there exist $m_1,m_2\in\N^*$ and $C_3,C_4>0$ such that
\begin{equation}\label{diid2}\|A^\frac{s}{2}\psi\|_{L^2}^{m_1+m_2}\leq C_3\|\psi\|_{L^2}^{m_1}\|A^\frac{s_1+2}{2}\psi\|_{L^2}^{m_2} \ \ \ \ \ \Longrightarrow \ \ \ \ \|\psi\|_{(s)}^{m_1+m_2}\leq C_4\|\psi\|_{L^2}^{m_1}\|\psi\|_{(s_1+2)}^{m_2}.\end{equation}
	Finally, when $B:H^{s_1}_{\Gi}\longrightarrow H^{s_1}_{\Gi}$ with $s_1>0$ and $(A,B)$ admits a 
non-degenerate chain of connectedness, the identities \eqref{casinooo}, \eqref{diid1} and \eqref{diid2} ensure the 
global 
approximate controllability in $H^s_\Gi$ for $s\in [s_1,s_1+2)$.

\smallskip

\noindent
{\bf 1) (c)  Conclusion.} Let $d$ be the parameter introduced by the validity of Assumptions II. If $d<2$, then 
$B:H^{2}_{\Gi}\rightarrow H^{2}_{\Gi}$ and the global approximate controllability is verified in $H^{d+2}_{\Gi}$ 
since $d+2<4.$ 
If $d\in [2,5/2)$, then $B:H^{d_1}\rightarrow H^{d_1}$ with $d_1\in(d,5/2)$ from Assumptions II. Now, $H^{d_1}_{\Gi}=H^{d_1}\cap H^2_\Gi$, thanks to Proposition $\ref{bor}$, and $B:H^{2}_{\Gi}\rightarrow H^{2}_{\Gi}$ implies $B:H^{d_1}_{\Gi}\rightarrow H^{d_1}_{\Gi}$. 
	The global approximate controllability is verified in $H^{d+2}_{\Gi}$ since $d+2<d_1+2.$
If $d\in [5/2,7/2)$, then $B:H_{\NN\KK}^{d_1}\rightarrow H_{\NN\KK}^{d_1}$ for $d_1\in(d,7/2)$ and $H^{d_1}_{\Gi}=H^{d_1}_{\NN\KK}\cap H^2_\Gi$ from Proposition $\ref{bor}$. Now, $B:H^{2}_{\Gi}\rightarrow H^{2}_{\Gi}$ that implies $B:H^{d_1}_{\Gi}\rightarrow H^{d_1}_{\Gi}$. 
The global approximate controllability is verified in $H^{d+2}_{\Gi}$ since $d+2<d_1+2.$

\smallskip
\noindent
{\bf 2) Generalization.} Let $(A,B)$ do not admit a non-degenerate chain of connectedness. We decompose 
$$A+u(\cdot)B=(A+u_0B)+u_1(\cdot)B,\ \ \ \ \ \ \  \ \ \ \ \ \  \ u_0\in \R,\ \ \ \ u_1\in L^2((0,T),\R).$$ If 
$(A,B)$ satisfies Assumptions I$(\eta)$ and Assumptions II$(\eta,\tilde d)$ for $\eta>0$ and $\tilde d\geq 0$, then 
Lemma \ref{chain1} and Lemma \ref{equi1} are valid. We consider $u_0$ in the neighborhoods provided by the two 
lemmas and we denote $(\phi_k^{u_0})_{k\in\N^*}$ a Hilbert basis of $\Hi$ made by eigenfunctions of $A+u_0B$. The 
point {\bf 1)\,\,}can be repeated by considering the sequence $(\phi_k^{u_0})_{k\in\N^*}$ instead of 
$(\phi_k)_{k\in\N^*}$ and the spaces $D(|A+u_0B|^\frac{s}{2})$ in substitution of $H^{s}_\Gi$ with $s>0$. The claim 
is equivalently proved since $(A+u_0B,B)$ admits a non-degenerate chain of connectedness from Lemma 
\ref{chain1} and $\big\||A+u_0B|^\frac{s}{2}\cdot\big\|_{L^2}\asymp\|\cdot\|_{(s)}$ with $s=2+d$ and $d$ from 
Assumptions II$(\eta,\tilde d)$ thanks to Lemma \ref{equi1}. \qedhere
\end{proof}

\subsection{Proof of Theorem $\ref{global}$}

Let $T,\epsilon>0$ be so that Proposition \ref{lcl} is valid. Let us assume $\psi_1,\psi_2\in H^{s}_{\Gi}$ such 
that $\|\psi_1\|_{L^2}=\|\psi_2\|_{L^2}=1$.  The same technique also applies in the general case. Thanks to 
Proposition \ref{approx}, we have $$\exists T_1,T_2>0,\ u_1\in L^2((0,T_1),\R),\ u_2\in 
L^2((0,T_2),\R)\ \ \ \ :\ \ \ \ \|\G^{u_1}_{T_1}\psi_1-\phi_1\|_{(s)}<{\epsilon},\ \ \ \ \ 
\|\G^{u_2}_{T_2}\psi_2-\phi_1\|_{(s)}<{\epsilon}$$
	and then $ \G^{u_1}_{T_1}\psi_1,\G^{u_2}_{T_2}\psi_2\in O_{\epsilon,T}^{s}.$ From {\bf 1)}, there exist 
$u_3,u_4\in L^2((0,T),\R)$ such that $$\G_T^{u_3}\G^{u_1}_{T_1}\psi_1=\G_T^{u_4}\G^{u_2}_{T_2}\psi_2=\phi_1\ \ \ 
\ 
 \Longrightarrow \ \ \ \ \ \exists T>0,\ \widetilde u\in L^2((0,\widetilde T),\R)\ \ :\ \ \G_{\widetilde 
T}^{\widetilde u}\psi_1=\psi_2.$$ 

\needspace{4mm}
{\noindent \bf{Acknowled}g{ments}.} The author would like to thank the
referees for the constructive comments which improved the organization of the work. He is also grateful to Olivier 
Glass and Nabile Boussa\"id for having carefully reviewed the presented theory and to 
Ka\"is Ammari for suggesting him the problem. He also thanks the colleagues Riccardo Adami, 
Enrico Serra and Paolo Tilli for the fruitful conversations.

\appendix\section{Appendix: Some auxiliary spectral results}\label{numeri}
In the current appendix, we characterize $(\lambda_k)_{k\in\N^*}$, the eigenvalues of the Laplacian $A$ in the \eqref{mainx1}, according to the structure of $\Gi$ and to the definition of $D(A)$.

\begin{prop}{(Roth's Theorem; $\cite{Roth}$)}\label{rotth}
	If $z$ is an algebraic irrational number, then for every $\epsilon>0$ the inequality $\big|z-\frac{n}{m}\big|\leq\frac{1}{m^{2+\epsilon}}$ is satisfied for at most a finite number of $n,m\in\Z.$
\end{prop}

\begin{proof}[Proof of Lemma \ref{figo}]
	For every $k\in\N^*$, there exist $m, n\in \N^*$, $i\leq N_1$ and $l\leq N_2$ such that $\lambda_{k+1}^1=\frac{m^2\pi^2}{L_l^2}$ and $\lambda_k^2=\frac{n^2\pi^2}{\tilde L_i^2}$. We suppose $L_l<\tilde L_i$. Let $z$ be an algebraic irrational number. From Proposition $\ref{rotth}$, we have that, for every $\epsilon>0$, there exists $C>0$ such that $|z-{n}/{m}|\geq{C}{m^{-2-\epsilon}}$ for every $m,n\in\N^*.$ Thus, when $m<n$, for each $\epsilon>0$, there exists $C_1>0$ such that
	\begin{equation*}\begin{split}
	\Big|\frac{m^2\pi^2}{L_l^{2}}-\frac{n^2\pi^2}{\tilde L_i^{2}}\Big|&=
	\Big|\Big(\frac{m\pi}{L_l}+\frac{n\pi}{\tilde L_i}\Big)
	\Big(\frac{m\pi}{L_l}-\frac{n\pi}{\tilde L_i}\Big)\Big|\geq  \frac{2m\pi}{\tilde L_i}\Big|\frac{m\pi}{L_l}-\frac{n\pi}{\tilde L_i}\Big|\geq
	\frac{2C_1\pi^2}{m^{\epsilon}\tilde L_i^2}.\\
	\end{split}\end{equation*}

	\noindent
	If $m\geq n$, then $\big|\frac{m^2\pi^2}{L_l^{2}}-\frac{n^2\pi^2}{\tilde L_i^{2}}\big|\geq \pi^2({L_l^{-2}-\tilde L_i^{-2}})$, which conclude the proof.
\end{proof}
%

We consider now the techniques developed in $\cite[Appendix\ A]{wave}$ in order to prove $\cite[Proposition\ 
A.11]{wave}$. For $x\in\R$, we denote by $E(x)$ the closest integer number to $x$ and  
$$\iii x\iii=\min_{z\in\Z}|x-z|,\ \ \ F(x)=x-E(x).$$ We notice $|F(x)|=\iii x\iii$ and $-\frac{1}{2}\leq F(z)\leq 
\frac{1}{2}.$ Let $\{L_j\}_{j\leq N}\in(\R^+)^N$ and $i\leq N$. We also define
$$n(x):=E\Big(x-\frac{1}{2}\Big),\ \ \ \  r(x):=F\Big(x-\frac{1}{2}\Big),\ \ \ \  d(x):=\iii x-\frac{1}{2}\iii ,\ \ \ \ \widetilde m^i(x) :=n\Big(\frac{L_i}{\pi}x\Big).$$

\begin{prop}\label{zuazua}
	Let $\{L_k\}_{k\leq N}\in\AL\LL(N)$ with $N\in\N^*$. Let $(\omega_n)_{n\in\N^*}$ be the unbounded ordered sequence of positive solutions of the equation 
	\begin{equation}\label{zizi}\sum_{l\leq N}\sin(x L_l)\prod_{m\neq l }\cos(x L_m)=0,\ \ \ \ \ \ \ \ \ x\in\R.\end{equation}
	For every $\epsilon>0$, there exists $C_\epsilon>0$ so that $|\cos(\omega_n L_l)|\geq \frac{C_\epsilon}{\omega_n^{1+\epsilon}}$ for every $l\leq N$ and $n\in\N^*.$
\end{prop}

\begin{proof}
From $\cite[relation\ (A.3)]{wave}$, for every $x\in\R$, we obtain the identities
	\begin{equation}\label{primom}2d({x})\leq|\cos(\pi x)|\leq\pi d({x}),\ \ \ \ \ \ \ \ 2d\Big( \Big(\widetilde m^i(x)+\frac{1}{2}\Big)\frac{L_j}{L_i}\Big)  \leq \Big|\cos\Big(\Big(\widetilde m^i(x)+\frac{1}{2}\Big)\frac{L_j}{L_i}\pi\Big)\Big|.\end{equation}
	As $\cos(\alpha_1-\alpha_2)=\cos(\alpha_1)\cos(\alpha_2)+\sin(\alpha_1)\sin(\alpha_2)$ for $\alpha_1,\alpha_2\in\R$ and $\widetilde m^i(x)+\frac{1}{2}=\frac{L_i}{\pi}x-r\big(\frac{L_i}{\pi}x\big)$ for every $x\in\R$, we have 
	\begin{equation}\begin{split}\label{bingo11}
	&2d\Big( \Big(\widetilde m^i(x)+\frac{1}{2}\Big)\frac{L_j}{L_i}\Big)  \leq 
	|\cos(L_jx)|+\left|\sin\left(\pi\frac{L_j}{L_i}\Big|r\Big(\frac{L_i}{\pi}x \Big)\Big|\right)\right|.\\
	\end{split}\end{equation}
	From $\cite[relation\ (A.3)]{wave}$ and $(\ref{primom})$, we have the following inequalities $|\sin(\pi|r(\cdot)|)|\leq\pi\iii | r(\cdot)|\iii\leq \pi|r(\cdot)|= \pi d(\cdot)\leq \frac{\pi}{2}|\cos(\pi (\cdot))|$, which imply $\big|\sin\big(\pi\frac{L_j}{L_i}\Big|r\big(\frac{L_i}{\pi}x \big)\big|\big)\big|\leq \pi\frac{L_j}{L_i}\big|r\big(\frac{L_i}{\pi}x \big)\big|\leq \frac{\pi L_j}{2 L_i}|\cos(L_ix)|$ for every $x\in\R.$
From $(\ref{bingo11})$, there exists $C_1>0$ such that, for every $i\leq N$,
	$$\prod_{j\neq i}d\Big( \Big(\widetilde m^i(x)+\frac{1}{2}\Big)\frac{L_j}{L_i}\Big) \leq \frac{1}{2^{N-1}} \prod_{j\neq i}|\cos(L_jx)|+C_1|\cos(L_ix)|\ \ \ \forall x\in\R.$$
	
%
%
%

\noindent	
If there exists $(\omega_{n_k})_{k\in\N^*}\subseteq (\omega_n)_{n\in\N^*}$ such that $|\cos(L_j\omega_{n_k})|\xrightarrow{k\rightarrow\infty} 0$, then $\prod_{j\neq i}|\cos(L_i\omega_{n_k})|\xrightarrow{k\rightarrow\infty}  0$ thanks to $(\ref{zizi})$. Equivalently to $\cite[relation\ (A.10)]{wave}$ (proof of $\cite[Proposition\ A.11]{wave}$), there exists a constant $C_2>0$ such that, for every $i\in\{0,...,N\}$, we have
	$$ C_2|\cos(L_i\omega_n)| \geq \prod_{j\neq i}d\Big( \Big(\widetilde m^i(\omega_n)+\frac{1}{2}\Big)\frac{L_j}{L_i}\Big) =\prod_{j\neq i}\iii \frac{1}{2}\Big(\Big(\widetilde m^i(\omega_n)+\frac{1}{2}\Big)\frac{2L_j}{L_i}-1\Big)\iii.$$
	Now, we have $\iii \frac{1}{2}(\cdot)\iii\geq\frac{1}{2}\iii\cdot\iii $ and $\iii(\cdot)-1\iii=\iii\cdot\iii$. We consider the Schmidt's Theorem $\cite[Theorem\ A.7]{wave}$ since $\{L_k\}_{k\leq N}\in\AL\LL(N)$. For every $\epsilon>0$, there exist $C_3,C_4>0$ such that, for every $n\in\N^*$, we have $\prod_{j\neq i}\frac{1}{2}\iii \Big(\widetilde m^i(\omega_n)+\frac{1}{2}\Big)\frac{2L_j}{L_i}\iii\geq \frac{C_3}{(2\widetilde m^i(\omega_n)+{1})^{1+\epsilon}}\geq
 \frac{C_4}{\omega_n^{1+\epsilon}}.\qedhere$	\end{proof}
\begin{osss}\label{zuazuafigo}	The techniques proving $\cite[Proposition\ A.11]{wave}$ and Proposition $\ref{zuazua}$ lead to the following results. Let $(\omega_n)_{n\in\N^*}\subset\R^+$ be an unbounded sequence and $(\omega_{n_k})_{k\in\N^*}$ any subsequence of $(\omega_n)_{n\in\N^*}$. Let $\{L_k\}_{k\leq N}\in\AL\LL(N)$ with $N\in\N^*$ and $l\leq N$.

\smallskip
\noindent
{\bf 1)} If $|\cos(L_l\omega_{n_k})|\xrightarrow{k\rightarrow\infty} 0$ implies $\prod_{j\neq l}|\cos(L_j\omega_{n_k})|\xrightarrow{k\rightarrow\infty} 0$ or $\prod_{j\neq l}|\sin(L_j\omega_{n_k})|\xrightarrow{k\rightarrow\infty} 0,$ then
\begin{equation}\label{zuazuacos}\forall \epsilon>0,\ \ \ \ \exists C>0\ \ \ \ :\ \ \ \ |\cos(\omega_n L_l)|\geq {C}{\omega_n^{-1-\epsilon}},\ \ \ \ \ \ \ \  \forall l\leq N,\ n\in\N^*.\end{equation}

\noindent
{\bf 2)} If $|\sin(L_l\omega_{n_k})|\xrightarrow{k\rightarrow\infty} 0$ implies $\prod_{j\neq l}|\cos(L_j\omega_{n_k})|\xrightarrow{k\rightarrow\infty} 0$ or $\prod_{j\neq l}|\sin(L_j\omega_{n_k})|\xrightarrow{k\rightarrow\infty} 0,$ then
\begin{equation}\label{zuazuasin}\forall \epsilon>0,\ \ \ \ \exists C>0\ \ \ \ :\ \ \ \ |\sin(\omega_n L_l)|\geq {C}{\omega_n^{-1-\epsilon}},\ \ \ \ \ \ \ \  \forall l\leq N,\ n\in\N^*.\end{equation}

\end{osss}
\section{Appendix: Moment problem}\label{momenti}
Let $\Hi=L^2((0,T),\R)$ with $T>0$ and $\Z^*=\Z\setminus\{0\}$. Let ${\bf \Lambda}=(\lambda_k)_{k\in\Z^*}$ be pairwise distinct ordered real numbers such that
\begin{equation}\begin{split}\label{gapp11}
\exists \MM\in\N^*,\ \exists\delta>0\ \ \  :\ \ \ \inf_{\{k\in\Z^*\ :\ k+\MM\neq 0 \}}|\lambda_{k+\MM}-\lambda_k|\geq\delta \MM.\\
\end{split}\end{equation}
From \eqref{gapp11}, there do not exist $\MM$ consecutive $k\in\Z^*$ such that 
$|\lambda_{k+1}  -  \lambda_{k}| < \delta$ and then, there exist some $j\in\Z^*\setminus\{-1\}$ such that 
$|\lambda_{j+1}-\lambda_{j}|\geq\delta$. This leads to a partition of $\Z^*$ in subsets $\{E_m\}_{m\in\Z^*}$ that we 
construct as follows. We denote by 
$(l_m)_{m\in\Z^*}\subseteq\Z^*\setminus\{-1\}$ the ordered
sequence of all the numbers such that $|\lambda_{l_m+1}-\lambda_{l_m}|\geq\delta.$ We add the value $-1$ when 
$|\lambda_{1}-\lambda_{-1}|\geq\delta$. We denote by
$\{E_{m}\}_{m\in\Z^*}$ the sets
\begin{equation*}\begin{split}
E_{-1}=\Big\{k\in \Z^*\ :\ l_{-1}+1\leq k\leq 
l_{1}\Big\},\ \ \ \ \ \  \ \ E_m=\Big\{k\in \Z^*\ :\ l_{m}+1&\leq k\leq l_{m+1}\Big\}
\end{split}\end{equation*}
with $m\in\Z^*\setminus\{-1\}$. The partition of $\Z^*$ in subsets $\{E_m\}_{m\in\Z^*}$ also defines an 
equivalence 
relation in $\Z^*$. Now, 
$\{E_m\}_{m\in\Z^*}$ are the equivalence classes corresponding to such relation and $|E_m|\leq \MM$ thanks to 
\eqref{gapp11}. Let $s(m)$ 
be the smallest element of $E_m$. For every ${\bf {x}}:=(x_k)_{k\in\Z^*}\subset\C$ and $m\in\Z^*,$ we define 
$${\bf 
{x}}^m:=(x^m_l)_{l\leq |E_m|},\ \ \ \ \  : \ \ \  \ \ x_l^m=x_{s(m)+(l-1)},\ \ \ \ \ \forall l\leq |E_m|.$$
In other words, ${\bf x}^m$ is the vector in $\C^{|E_m|}$ composed by those elements of ${\bf x}$ with indices in 
$E_m$.
For every $m\in\Z^*$, we denote $F_m({\bf {\bf \Lambda}}^m):\C^{|E_m|}\rightarrow \C^{|E_m|}$ the matrix with 
components
\begin{equation*}
\begin{split}
F_{m;j,k}({\bf {\bf \Lambda}}^m):=\begin{cases}
\prod_{\underset{ l\leq k}{l\neq j}}(\lambda^m_j-\lambda_l^m)^{-1},\ \ \ \ \ & j\leq k,\\
1,\ \ \ \ \ \ \ \ \ \ \ \ \ \ \ \ & j=k=1,\\
0,\ \ \ \ \ \ \ \ \ \ \ \ \ \ \ \ & j>k,\\
\end{cases}\ \ \ \ \ \ \ \ \ \ \ \ \ \ \forall j,k\leq {|E_m|}.
\end{split}
\end{equation*}For each $k\in\Z^*$, there exists $m(k)\in\Z^*$ such that $k\in E_{m(k)}$, while $s(m(k))$ 
represents the smallest element of $E_{m(k)}$. Let $F({\bf {\bf \Lambda}})$ be the infinite matrix acting on ${\bf 
x}=(x_k)_{k\in\Z^*}\subset\C$ as follows
$$\big(F({\bf {\bf \Lambda}}){\bf x}\big)_k=\Big(F_{m(k)}({\bf {\bf \Lambda}}^{m(k)}){\bf 
x}^{m(k)}\Big)_{k-s(m(k))+1},\ \ \ \ 
\  \ \ \ \ \ \forall\, k\in\Z^*.$$
We consider $F({\bf {\bf \Lambda}})$ as the operator on $\ell^2(\Z^*,\C)$ defined by the action\ above and with 
domain
$$H({\bf {\bf \Lambda}}):=D(F({\bf {\bf \Lambda}}))=\left\{{\bf  x}:=(x_k)_{k\in\Z^*}\in\ell^2(\Z^*,\C)\ :\ F({\bf 
{\bf \Lambda}}){\bf  x}\in\ell^2(\Z^*,\C)\right\}.$$

\begin{osss}\label{limitatelo}
Each matrix $F_m({\bf {\bf \Lambda}}^m)$ with $m\in\Z^*$ is invertible and we call $F_m({\bf {\bf 
\Lambda}}^m)^{-1}$ its 
inverse. Now, $F({\bf {\bf \Lambda}}):H({\bf {\bf \Lambda}})\rightarrow Ran(F({\bf {\bf \Lambda}}))$ is invertible 
and $F({\bf 
{\bf \Lambda}})^{-1}:Ran(F({\bf {\bf \Lambda}}))\rightarrow H({\bf {\bf \Lambda}})$ is so that, for ${\bf x}\in 
Ran(F({\bf 
{\bf \Lambda}}))$,
$$\big(F({\bf {\bf \Lambda}})^{-1}{\bf x}\big)_k=\Big(F_{m(k)}({\bf {\bf \Lambda}}^{m(k)})^{-1}{\bf 
x}^{m(k)}\Big)_{k-s(m(k))+1},\ \ \ \ \  \ \ \ \ \forall k\in\Z^*.$$
\end{osss}

Let $F_{m(k)}({\bf {\bf \Lambda}}^{m(k)})^*$ be the transposed matrix of $F_{m(k)}({\bf {\bf \Lambda}}^{m(k)})$ 
for 
every 
$m\in\Z^*$. Let $F({\bf {\bf \Lambda}})^*$ be the infinite matrix so that, for every ${\bf 
x}=(x_k)_{k\in\Z^*}\subset\C$, 
$$\big(F({\bf {\bf \Lambda}})^*{\bf x}\big)_k=\Big(F_{m(k)}({\bf {\bf \Lambda}}^{m(k)})^*{\bf 
x}^{m(k)}\Big)_{k-s(m(k))+1},\ \ 
\ \ \  \ \ \ \ \forall k\in\Z^*.$$

\begin{prop}\label{stima1}
	Let ${\bf \Lambda}:=(\lambda_k)_{k\in\Z^*}$ be an ordered sequence of real numbers satisfying \eqref{gapp11}. 
Sufficient condition to have $H({\bf \Lambda})\supseteq h^{\tilde d}(\C)$ is the existence of $\tilde d\geq 0$ and 
$C>0$ such that
	\begin{equation}\begin{split}\label{gapp3}
	|\lambda_{k+1}-\lambda_k|\geq C |k|^{-\frac{\tilde d}{\MM-1}}\ \ \ \ \ \forall k\in\Z^*.\\
	\end{split}\end{equation}
\end{prop}
\begin{proof}
	Thanks to $(\ref{gapp3})$, we have $|\lambda_j-\lambda_k|\geq C\min_{l\in E_m}|l|^{-\frac{\tilde d}{\MM-1}}$ 
for every $m\in\Z^*$ and $j,k\in E_m$.
	There exists $C_1>0$ such that, for $1<j,k\leq i(m)$,
	\begin{equation*}\begin{split}
	&|F_{m;j,k}({\bf \Lambda}^m)|\leq C_1\big(\max_{l\in E_m}|l|^{\frac{\tilde d}{\MM-1}}\big)^{k-1}\leq 
C_1\big(\max_{l\in E_m}|l|^{\frac{\tilde d}{\MM-1}}\big)^{\MM-1}\leq C_1 2^{\MM{\tilde d}}\min_{l\in 
E_m}|l|^{{\tilde d}},\ \ \ \ \ \ \ \  |F_{m;1,1}({\bf \Lambda}^m)|=1.
	\end{split}\end{equation*}
There exist $C_2,C_3>0$ such that, for $j\leq i(m)$, we have $\big(F_{m}({\bf \Lambda}^m)^*F_{m}({\bf 
\Lambda}^m)\big)_{j,j}\leq
	C_2\min_{l\in E_m}|l|^{2{\tilde d}}$ and then $Tr\Big(F_{m}({\bf \Lambda}^m)^*F_{m}({\bf 
\Lambda}^m)\Big)\leq C_3\min_{l\in E_m}|l|^{2{\tilde d}}$. Let $\rho(M)$ be the 
spectral radius of a matrix M and we denote $\iii M\iii=\sqrt{\rho(M^*M)}$ its euclidean norm. As $\big(F_{m}({\bf 
\Lambda}^m)^*F_{m}({\bf \Lambda}^m)\big)$ is positive-definite, there holds
	$$\iii F_m({\bf \Lambda}^m)\iii^2=\rho\big(F_{m}({\bf \Lambda}^m)^*F_{m}({\bf \Lambda}^m)\big)\leq
	C_3\min_{l\in E_m}|l|^{2{\tilde d}},\ \ \ \ \forall m\in\Z^*.$$
	In conclusion, $\|F({\bf \Lambda})
	{\bf x}\|_{\ell^2}^2\leq C_3\|{\bf x}\|_{h^{\tilde d}}^2<+\infty$ for ${\bf x}=(x_k)_{k\in\Z^*}\in h^{\tilde 
d}(\C)$ as
\begin{equation*}\begin{split}
	&\|F({\bf \Lambda})
	{\bf x}\|_{\ell^2}^2\leq\sum_{m\in \Z^*}\iii F_m({\bf \Lambda}^m)\iii
	^2\sum_{l\in E_m} |x_l|^2\leq C_3\sum_{m\in \Z^*}\min_{l\in E_m}|l|^{2{\tilde d}}
	\sum_{l\in E_m} |x_l|^2.\qedhere
	\end{split}
	\end{equation*}
\end{proof}

\begin{osss}\label{aggiunto}
	Thanks to Proposition $\ref{stima1}$, when $(\lambda_k)_{k\in\Z^*}$ satisfies $(\ref{gapp11})$ and 
$(\ref{gapp3})$, the space $H({\bf \Lambda})$ is dense in $\ell^2(\C)$ as $h^{\tilde d}$ is dense in $\ell^2$. 
Now, we consider $F({\bf {\bf \Lambda}})^*$ as the unique 
adjoint operator of $F({\bf {\bf \Lambda}})$ in $\ell^2(\Z^*,\C)$ with domain $H({\bf {\bf \Lambda}})^*:=D(F({\bf 
{\bf \Lambda}})^*)$.
	As in Remark $\ref{limitatelo}$, we define $(F({\bf {\bf \Lambda}})^*)^{-1}$ the inverse of $F({\bf 
{\bf \Lambda}})^*:H({\bf {\bf \Lambda}})^*\rightarrow Ran(F({\bf {\bf \Lambda}})^*)$ and $(F({\bf {\bf 
\Lambda}})^*)^{-1}=(F({\bf 
{\bf \Lambda}})^{-1})^*$. Finally, $H({\bf \Lambda})^*\supseteq h^{\tilde d}(\C)$ which follows as 
Proposition \ref{stima1}.
	\end{osss}

Let ${\bf e}$ be the sequence of functions in $L^2((0,T),\C)$ with $T>0$ so that
${\bf e}:=(e^{i\lambda_k (\cdot)})_{k\in\Z^*}.$
We denote by ${\bf \Xi}$ the so-called {divided differences} of the family $(e^{i\lambda_kt})_{k\in\Z^*}$ such that
\begin{align*}\label{xi_riesz}{\bf \Xi}:=(\xi_k)_{k\in\Z^*}=F({\bf {\Lambda}})^*{\bf e}.\end{align*}
In the following theorem, we rephrase a result of Avdonin and Moran $\cite{avdd}$, which is also proved by Baiocchi, Komornik and Loreti in $\cite{balocchi}$.
\begin{teorema}[Theorem\ 3.29; $\cite{wave}$]\label{riesz}
	Let $(\lambda_k)_{k\in\Z^*}$ be an ordered sequence of pairwise distinct real numbers satisfying $(\ref{gapp11})$.
	If $T>2\pi/\delta$, then $(\xi_k)_{k\in\Z^*}$ forms a Riesz Basis in the space $X:=\overline{span\{\xi_k|\ k\in\Z^*\}}^{ L^2}.$
\end{teorema}

\begin{prop}\label{momentproblem123}
	Let $(\omega_k)_{k\in\N^*}\subset\R^+\cup\{0\}$ be an ordered sequence of real numbers with $\omega_1=0$ such that there exist $\tilde d\geq 0$, $\delta,C>0$ and $\MM\in\N^*$ with
	\begin{equation}\label{numer}\begin{split}
	\inf_{{k\in\N^*}}|\omega_{k+\MM}-\omega_k|\geq\delta \MM,\ \ \ \ \ \  
	|\omega_{k+1}-\omega_k|\geq C k^{-\frac{\tilde d}{\MM-1}},\ \ \ \ \ \forall k\in\N^*.\\
		\end{split}\end{equation}
	Then, for $T>2\pi/\delta$ and for every $(x_k)_{k\in\N^*}\in h^{\tilde d}(\C)$ with $x_1\in\R$, 
	\begin{equation}\label{momem}
	\exists u\in L^2((0,T),\R)\ \ \ \ :\ \ \ \ {x_{k}}=\int_{0}^Tu(\tau)e^{i\omega_k \tau}d\tau\ \ \ \ \ \ \ \ \forall k\in\N^*.\end{equation}
\end{prop}
\begin{proof}
When $k> 0$, we call $\lambda_k=\omega_{k}$, while $\lambda_k=-\omega_{-k}$ for $k<0$ such that ${k}\neq -1$. The 
sequence $(\lambda_k)_{k\in \Z^*\setminus\{-1\}}$ satisfies $(\ref{gapp11})$ and \eqref{gapp3} with respect to the 
indices $\Z^*\setminus\{-1\}$.
Theorem \ref{riesz} and the properties of a Reisz basis (see for instance $\cite[Appendix\ B.1;\ Definition\ 2\ 
\&\ Proposition\ 19)]{laurent}$ ensure the invertibility of the map $$M:g\in X\mapsto (\la 
\xi_k,g\ra_{L^2((0,T),\C)})_{k\in\Z^*\setminus\{-1\}}\in \ell^2(\C) ,\ \ \  \ \ \ \text{with}\ \ \ \  \ \ \ \la 
\xi_k,g\ra_{L^2((0,T),\C)}=(F({\bf \Lambda})^*\la {\bf e },g\ra_{L^2((0,T),\C)})_k.$$ Now, 
$H({\bf \Lambda})^*\supseteq h^{\widetilde d}(\C)$ from Remark 
$\ref{aggiunto}$ and the following map is invertible 
	$$ (F({\bf \Lambda})^*)^{-1}\circ M:g\in \tilde{X}\mapsto (\la 
e^{i\omega_k(\cdot)},g\ra_{L^2((0,T),\C)})_{k\in\Z^*\setminus\{-1\}}\in h^{\tilde d}(\C),\ \ \ \ \ \  \ 
\tilde{X}:=M^{-1}\circ 
F({\bf \Lambda})^*(h^{\widetilde d}(\C)).$$
For every $(\tilde x_k)_{k\in\Z^*\setminus\{-1\}}\in h^{\tilde d}(\Z^*\setminus\{-1\},\C)$, there exists $u\in 
L^2((0,T),\C)$ so that $\tilde {x_{k}}=\int_{0}^Tu(\tau)e^{i\lambda_k \tau}d\tau$ for every $k\in\Z^*.$ 
		 Given $(x_k)_{k\in\N^*}\in h^{\tilde d}(\N^*,\C)$, we call $(\tilde x_k)_{k\in 
\Z^*\setminus\{-1\}}\in h^{\tilde d}(\Z^*\setminus\{-1\},\C)$ such that $\tilde x_k=x_{k}$ for $k> 0$, while 
$\tilde x_k=\overline x_{-k}$ for $k<0$ and $k\neq -1$. As above, there exists $u\in L^2((0,T),\C)$ so that 
$x_1=\int_0^Tu(s)ds$ and $$\tilde x_k=\int_0^Tu(s)e^{-i\lambda_k s}ds, \ \ \forall k\in \Z^*\setminus\{-1\}\ \ \ \  
\ 
	\Longrightarrow\ \ \ \  $$
	$$\int_0^Tu(s)e^{i\omega_k s}ds=x_k=\int_0^T\overline{u}(s)e^{i\omega_k s}ds,\ \ \ \ \ \ \ \ k\in\N^*\setminus\{1\},$$
$$\int_0^T\overline{u}(s)e^{-i\omega_k s}ds=\overline{x_k}=\int_0^Tu(s)e^{-i\omega_k s}ds,\ \ \ \ \ \ \ \ k\in\N^*\setminus\{1\}.$$	
Then, $(\ref{momem})$ is solvable for $u\in L^2((0,T),\C)$ and we just need to prove that $u$ is actually real. The last relations and $x_{1}\in\R$ imply $\la 
Imm(u),e^{i\lambda_k(\cdot)}\ra_{L^2(0,T)}=0$ for every $k\in\Z^*\setminus\{-1\}$ and then 
$$\la Imm(u),\xi_k\ra_{L^2(0,T)}=\Big(F({\bf \Lambda})^*\big(\la 
Imm(u),e^{i\lambda_l(\cdot)}\ra_{L^2(0,T)}\big)_{l\in\Z^*\setminus\{-1\}}\Big)_k=0,\ \ \ \  \ \ \forall 
k\in\Z^*\setminus\{-1\}.$$
Now, the family $\{e^{i\lambda_kt}\}_{k\in\Z^*\setminus\{-1\}}$ is minimal in $X$ and 
$X=\overline{\spn\{e^{i\lambda_kt}:k\in\Z^*\setminus\{-1\}\}}^{ L^2}.$ For every 
$u\in X$,
$$\overline{u}\in \overline{\spn\{e^{-i\lambda_kt}:k\in\Z^*\setminus\{-1\}\}}^{ 
L^2}=\overline{\spn\{e^{i\lambda_kt}:k\in\Z^*\setminus\{-1\}\}}^{ 
L^2}=X,$$
which implies that $Imm(u)=\frac{u-\overline{u}}{2i}\in X$. We call ${\bf {v}}:=( v_k)_{k\in\Z^*\setminus\{-1\}}$ 
the biorthogonal sequence to ${\bf  \Xi}$ in $X$ which is also a Riesz basis of $X$. For every $u\in X$, we have 
$Imm(u)\in X$ and \begin{align*}Imm(u)=\sum_{k\in\Z^*\setminus\{-1\}}v_k c_{k},\ \ \ \  
\ \text{for}\ \ \ \ \ c_k=\la Imm(u),\xi_k\ra_{L^2(0,T)},\ \ \ \ \forall k\in\Z^*\setminus\{-1\}.\end{align*} In 
conclusion, when the function $u$ verifies $\la Imm(u),\xi_k\ra_{L^2(0,T)}=0$ for every $k\in\Z^*\setminus\{-1\}$, we have $Imm(u)=0$ and then $(\ref{momem})$ is solvable for $u\in L^2((0,T),\R)$.	\qedhere
\end{proof}

\begin{prop}\label{uppp123}
	Let $(\lambda_k)_{k\in\Z^*}$ be an ordered sequence of pairwise distinct real numbers satisfying $(\ref{gapp11})$. For every $T>0$, there exists $C(T)>0$ uniformly bounded for $T$ lying on bounded intervals such that
	$$\forall g\in L^2((0,T),\C),\ \ \ \ \ \  \ \left\|\int_0^T 
	e^{i\lambda_{(\cdot)} s}g(s) ds\right\|_{\ell^2}\leq C(T)\|g\|_{L^2((0,T),\C)}.$$
\end{prop}
\begin{proof}	{\bf 1) Uniformly separated numbers.} Let $(\omega_k)_{k\in\Z^*}\subset\R$ be such that 
$\gamma:=\inf_{k\neq j}|\omega_k-\omega_j|>0$. In the current proof, we adopt the notation $L^2:=L^2((0,T),\C)$.
	Thanks to the Ingham's Theorem $\cite[Theorem \ 4.3]{Ing}$, the sequence $\{e^{i\omega_k 
(\cdot)}\}_{k\in\Z^*}$ is a Riesz Basis in 
	$X=\overline{span\{e^{i\omega_k (\cdot)}:\ k\in\Z^*\}}^{\ L^2}\subset L^2$ when $T>{2\pi}/{\gamma}.$
 Now, there exists $C_1(T)>0$ such that $\sum_{k\in \Z^*}|\la e^{i\omega_k (\cdot)},u\ra_{L^2}|^2\leq 
C_1(T)^2\|u\|_{L^2}^2$ for every $ u\in X$ as in $\cite[relation\ (30)]{mio2}$. Let $P:L^2\longrightarrow X$ be 
the orthogonal projector. For $g\in L^2$, we have
	\begin{equation*}\begin{split}&\big\|(\la e^{i\omega_k 
(\cdot)},g\ra_{L^2})_{k\in\Z^*}\big\|_{\ell^2}=\big\|(\la e^{i\omega_k 
(\cdot)},Pg\ra_{L^2})_{k\in\Z^*}\big\|_{\ell^2}\leq C_1(T)\|Pg\|_{L^2}\leq 
C_1(T)\|g\|_{L^2}.\end{split}\end{equation*}

	\noindent
	{\bf 2) Pairwise distinct numbers.} Let $(\lambda_k)_{k\in\Z^*}$ be as in the hypotheses. We decompose 
$(\lambda_k)_{k\in\Z^*}$ in $\MM$ sequences $(\lambda_k^j)_{{k\in\Z^*}}$ with $j\leq \MM$ such that $\inf_{k\neq 
l}|\lambda_k^j-\lambda_l^j|\geq\delta\MM$ for every $j\leq \MM.$ Now, for every $j\leq \MM$, we apply the point 
{\bf 1)} with $(\omega_k)_{k\in\Z^*}=(\lambda_k^j)_{{k\in\Z^*}}.$ For every $T>2\pi/\delta\MM$ and $g\in L^2$, 
there exists $C(T)>0$ uniformly bounded for $T$ in bounded intervals such that
	\begin{equation*}\begin{split}
	&\left\|(\la e^{i\lambda_k (\cdot)},g\ra_{L^2})_{k\in\Z^*}\right\|_{\ell^2}\leq\sum_{j=1}^{\MM} \left\|(\la 
e^{i\lambda_k^j (\cdot)},g\ra_{L^2})_{k\in\Z^*}\right\|_{\ell^2}\leq \MM 
C(T)\|g\|_{L^2},\end{split}.\end{equation*}

	\noindent
	{\bf 3) Conclusion.} We know $\big\|\int_0^T e^{i\lambda_{(\cdot)} \tau}{g}(\tau)dt\big\|_{\ell^2}\leq \MM C(T)\|{g}\|_{L^2}$  for every $g\in L^2$ and, for $T>2\pi/\delta\MM$, we choose the smallest value possible for $C(T)$. When $T\leq 2\pi/\delta\MM$, for $g\in L^2$, we define $\tilde g\in L^2((0,2\pi/\delta\MM+1),\C)$ such that $\tilde g = g$ on $(0,T)$ and $\tilde g=0$ in $(T, 2\pi/\delta\MM+1)$. Then
	\begin{equation*}\Big\|\int_0^T e^{i\lambda_{(\cdot)} \tau}{ 
g}(\tau)dt\Big\|_{\ell^2}=\Big\|\int_0^{2\pi/\delta\MM+1} e^{i\lambda_{(\cdot)} \tau}{\tilde 
g}(\tau)dt\Big\|_{\ell^2}\leq \MM C(2\pi/\delta\MM+1)\|{g}\|_{L^2}.\end{equation*}
	Let $0<T_1<T_2 <+\infty$, $g\in L^2(0,T_1)$ and $\tilde g \in L^2(0,T_2)$ be defined as $\tilde g = g$ on $(0,T_1)$ and $\tilde g=0$ on $(T_1,T_2)$. We apply the last inequality to $\tilde g$ that leads to $C(T_1)\leq C(T_2)$.\qedhere
\end{proof}

\section{Analytic perturbation}\label{analitics}
	The aim of the appendix is to adapt the perturbation theory results from $\cite[Appendix\ B]{mio1}$, 
where the $(\ref{mainx1})$ is considered on $\Gi=(0,1)$ and $A$ is the Dirichlet Laplacian. As in such 
work, we decompose $u(t)=u_0+u_1(t),$ for $u_0$ and $u_1(t)$ real. Let $A+u(t)B=A+u_0B+u_1(t)B$. We 
consider $u_0B$ as a perturbative term of $A$. Let $(\lambda_j^{u_0})_{j\in\N^*}$ be the ordered spectrum of 
$A+u_0B$ corresponding to some eigenfunctions $(\phi_j^{u_0})_{j\in\N^*}$.

Let the definition of 
$\{E_m\}_{m\in\Z^*}$ provided in the first part of Appendix $\ref{momenti}$. We repeat the construction of such 
equivalence classes by considering $(\lambda_j)_{j\in\N^*}$ the sequence of the eigenvalues  of $A$ in the 
\eqref{mainx1}. In this case, we consider the indices $\N^*$ instead of $\Z^*$ and the validiy of Lemma 
\ref{g13} instead of \eqref{gapp11}. We denote $n:\N^*\rightarrow\N^*$, $s:\N^*\rightarrow \N^*$ and 
$p:\N^*\rightarrow \N^*$ those applications respectively mapping $j\in\N^*$ in $n(j),s(j),p(j)\in\N^*$ such that
	$$j\in E_{n(j)},\ \ \ \ \ \ \ \lambda_{s(j)}=\inf\{\lambda_k>\lambda_j\ |\ k\notin E_{n(j)}\},\ \ \ \ \ \  \ 
\lambda_{p(j)}=\sup\{k\in E_{n(j)}\}.$$

	\begin{lemma}\label{bound1}
		Let $(A,B)$ satisfy Assumptions I$(\eta)$ and Assumptions II$(\eta,\tilde d)$ for $\eta>0$ and $\widetilde d\geq 0$. There exists a neighborhood $U(0)$ of $u=0$ in $\R$ such that there exists $ c>0$ so that
			$$\iii(A+u_0B-\nu_k)^{-1}\iii\leq c,\ \ \ \ \ \  \ \nu_k:=(\lambda_{s(k)}-\lambda_{p(k)})/2,\ \ \ \ \ \forall u_0\in U(0),\ \forall k\in\N^*.$$
			Moreover, let $P_j^{\bot}$ be the projector onto $\overline{span\{\phi_m\ :\ m\not\in E_{n(j)}\}}^{\ 
L^2}$ with $j\in\N^*$. For $u_0\in U(0)$, the operator $(A+u_0P_{k}^{\bot}B-\lambda_k^{u_0})$ is invertible with 
bounded inverse from $D(A)\cap Ran(P_{k}^{\bot})$ to $Ran(P_{k}^{\bot})$ for every $ k\in\N^*$.
	\end{lemma}
	\begin{proof}
	 The proof exactly follows the ones of $\cite[Lemma\ B.2\ \&\ Lemma\ B.3]{mio1}$.\qedhere 
	\end{proof}

	\begin{lemma}\label{chain1}
		Let $(A,B)$ satisfy Assumptions I$(\eta)$ and Assumptions II$(\eta,\tilde d)$ for $\eta>0$ and $\widetilde 
d\geq 0$. There exists a neighborhood $U(0)$ of $u=0$ in $\R$ such that, up to a countable subset $Q$ and for every 
$(k,j),(m,n)\in I:=\{(j,k)\in(\N^*)^2:j< k\},\ (k,j)\neq(m,n)$, we have
		$$\lambda_k^{u_0}-\lambda_j^{u_0}-\lambda_m^{u_0}+\lambda_n^{u_0}\neq 0,\ \ \ \ \ \ \ \ \la\phi_k^{u_0},B\phi_j^{u_0}\ra_{L^2}\neq 0,\ \ \ \ \ \ \  \ \forall u_0\in U(0)\setminus Q.$$
		\end{lemma}
	\begin{proof}
		For $k\in\N^*$, we decompose $\phi_k^{u_0}=a_k\phi_k+\sum_{j\in E_{n(k)}^*}\beta_j^k\phi_j +\eta_k,$ where $a_k\in\C$, $\{\beta_j^k\}_{j\in\N^*}\subset\C$ and $\eta_k$ is orthogonal to $\phi_l$ for every $l\in E_{n(k)}$. Moreover, $\lim_{|u_0|\rightarrow 0}|a_k|=1$ and $\lim_{|u_0|\rightarrow 0}|\beta_j^k|=0$ for every $j,k\in\N^*$. We denote $E_{n(k)}^*:=E_{n(k)}\setminus\{k\}$ for every $k\in\N^*$ and
		\begin{equation*}
		\begin{split}
		&\lambda_k^{u_0}\phi_k^{u_0}=
		(A+u_0B)\Big(a_k\phi_k+\sum_{j\in E_{n(k)}^*}\beta_j^k\phi_j+\eta_k\Big)=a_k(A+u_0B)\phi_k+\sum_{j\in E_{n(k)}^*}\beta_j^k(A+u_0B)\phi_j+(A+u_0B)\eta_k.\\
		\end{split}\end{equation*}
		Now, Lemma $\ref{bound1}$ leads to the existence of $C_1>0$ such that, for every $k\in\N^*$,
		\begin{equation}\begin{split}\label{33333}
		\eta_k=&-\big(\big(A+u_0P_{k}^{\bot}B-\lambda_k^{u_0}
		\big)P_{k}^{\bot}\big)^{-1}u_0
		\Big(a_kP_{k}^{\bot}B\phi_k+\sum_{j\in E_{n(k)}^*}\beta_j^kP_{k}^{\bot}B\phi_j\Big)
		\end{split}\end{equation}
		and $\|\eta_k\|_{L^2}\leq{C_1|u_0|}.$ 
		We compute $\lambda_k^{u_0}=\la\phi_k^{u_0},(A+u_0B)\phi_k^{u_0}\ra_{L^2}$ for every $k\in\N^*$ and
		\begin{equation*}\begin{split}
		&\lambda_k^{u_0}=\Big(\lambda_k|a_k|^2+\sum_{j\in E_{n(k)}^*}\lambda_j|\beta_j^k|^2\Big)
		+\la\eta_k,(A+u_0B)\eta_k\ra_{L^2}+u_0\sum_{j,l\in E_{n(k)}^*}\overline{\beta_j^k}\beta_l^kB_{j,l}\\
		&+u_0|a_k|^2B_{k,k}+2u_0 \Re\Big(\sum_{j\in E_{n(k)}^*}\beta_j^k\la\eta_k,B\phi_j\ra_{L^2}+\overline{a_k}
		\sum_{j\in E_{n(k)}^*}\beta_j^kB_{k,j}+\overline{a_k}\la \phi_k,B\eta_k\ra_{L^2}\Big).\\
		\end{split}\end{equation*}
		Thanks to $(\ref{33333})$, it follows $\la\eta_k,(A+u_0B)\eta_k\ra_{L^2}=O(u_0^2)$ for every $k\in\N^*$. 
Let $$\widehat a_k:=\frac{|a_k|^2+\sum_{j\in E_{n(k)}^*}|\beta_j^k|^2}{1-\|\eta_k\|_{L^2}^2},\ \ \ \  \ 
\ \ \widetilde a_k:=\frac{|a_k|^2+\sum_{j\in E_{n(k)}^*}\lambda_j/\lambda_k|\beta_j^k|^2}{1-\|\eta_k\|_{L^2}^2}.$$ 
As $\|\eta_k\|_{L^2}\leq{C_1|u_0|}$ for every $k\in\N^*$, it follows $\lim_{|u_0|\rightarrow 0}|\widehat a_k|=1$ 
uniformly in $k$. Thanks to the fact $\lim_{k\rightarrow +\infty}\inf_{j\in E_{n(k)}^*} 
{\lambda_j}{\lambda_k}^{-1}=\lim_{k\rightarrow +\infty}\sup_{j\in E_{n(k)}^*} {\lambda_j}{\lambda_k}^{-1}=1,$
we have $\lim_{|u_0|\rightarrow 0}|\widetilde a_k|=1$ uniformly in $k$.	Now, there exists $f_{k}$ such that $
		\lambda_k^{u_0}=\widetilde a_k\lambda_k+u_0\widehat a_kB_{k,k}+u_0f_k$ where $\lim_{|u_0|\rightarrow 0}f_{k}=0$ uniformly in $k$ (the relation is also valid when $\lambda_k=0$). For each $(k,j),(m,n)\in I$ such that $(k,j)\neq(m,n)$, there exists $f_{k,j,m,n}$ such that $\lim_{|u_0|\rightarrow 0}f_{k,j,m,n}=0$ uniformly in $k,j,m,n$ and
		\begin{equation*}
		\begin{split}
		&\lambda_k^{u_0}-\lambda_j^{u_0}-\lambda_m^{u_0}+\lambda_n^{u_0}=\widetilde a_k\lambda_k-\widetilde a_j\lambda_j-\widetilde a_m\lambda_m+\widetilde a_n\lambda_n+u_0f_{k,j,m,n}+u_0(\widehat a_kB_{k,k}-\widehat a_jB_{j,j}-\widehat a_mB_{m,m}\\
		&+\widehat a_nB_{n,n})=\widetilde a_k\lambda_k-\widetilde a_j\lambda_j-\widetilde a_m\lambda_m+\widetilde a_n\lambda_n+u_0(\widehat a_kB_{k,k}-\widehat a_jB_{j,j}-\widehat a_mB_{m,m}+\widehat a_nB_{n,n})+O(u_0^2).
		\end{split}
		\end{equation*}
		Thanks to the second point of Assumptions I, there exists $U(0)$ a neighborhood of $u=0$ in $\R$ small enough such that, for each $u\in U(0)$, we have that every function $\lambda_k^{u_0}-\lambda_j^{u_0}-\lambda_m^{u_0}+\lambda_n^{u_0}$ is not constant and analytic.
		Now, $V_{(k,j,m,n)}=\{u\in D\big|\ \lambda_k^{u}-\lambda_j^{u}-\lambda_m^{u}+\lambda_n^{u}=0\}$ is a discrete subset of $D$ and
		$$V=\{u\in D\big|\ \exists ((k,j),(m,n))\in I^2: \lambda_k^{u}-\lambda_j^{u}-\lambda_m^{u}+\lambda_n^{u}=0\}$$
		is a countable subset of $D$, which achieves the proof of the first claim. The second relation is proved with the same technique. For $j,k\in\N^*$, the analytic function $u_0\rightarrow \la\phi_j^{u_0},B\phi_k^{u_0}\ra_{L^2}$ is not constantly zero since $\la\phi_j,B\phi_k\ra_{L^2}\neq 0$ and $W=\{u\in D\big|\ \exists (k,j)\in I: \la\phi_j^{u_0},B\phi_k^{u_0}\ra_{L^2}=0\}$ is a countable subset of $D$. \qedhere
	\end{proof}
	

	\begin{lemma}\label{equi1}
		Let $(A,B)$ satisfy Assumptions I$(\eta)$ and Assumptions II$(\eta,\tilde d)$ for $\eta>0$ and $\widetilde 
d\geq 0$. Let $T>0$ and $d$ be introduced in Assumptions II. Let $c\in\R$ such that 
$0\not\in\sigma(A+u_0B+c)$ (the spectrum of $A+u_0B+c$) and such that $A+u_0B+c$ is a positive operator. There 
exists a neighborhood $U(0)$ of $0$ in $\R$ such that, for every $s\in (0,2+d]$,
\begin{equation}\label{cazr}\begin{split}\forall u_0\in 
U(0),\ \  \ \ \ \ \Big\||A+u_0B+c|^\frac{s}{2}\cdot\Big\|_{L^2}\ 
		\asymp\ \big\|\cdot\big\|_{(s)}.\end{split}\end{equation}\end{lemma}
	\begin{proof}
		Let $D$ be the neighborhood provided by Lemma $\ref{chain1}$. By applying the arguments of 
the proof of $\cite[Lemma\ B.6]{mio1}$, it is possible to prove that the relation $(\ref{cazr})$ is valid for 
$s\in [s_1,s_1+2)$ when $B:H^{s_1}_{\Gi}\longrightarrow H^{s_1}_{\Gi}$. By classical interpolation results, the 
relation $(\ref{cazr})$ is valid for every
$s\in [0,s_1+2)$. Finally, how to consider $s_1$ in the different cases of Assumptions II is treated by the 
point {\bf 2)} of the proof of Proposition \ref{approx}.\qedhere\end{proof}

\end{document}